\documentclass{amsart}

\usepackage[T1]{fontenc}
\usepackage{enumerate, amsmath, amsfonts, amssymb, amsthm, mathrsfs, wasysym, graphics, graphicx, xcolor, url, hyperref, hypcap, shuffle, xargs, multicol, overpic, pdflscape, multirow, hvfloat, minibox, accents, array, multido, xifthen, a4wide, ae, aecompl, blkarray, pifont, mathtools, etoolbox, dsfont, verbatim, stackrel, stmaryrd, subcaption}
\usepackage{marginnote}
\hypersetup{colorlinks=true, citecolor=darkblue, linkcolor=darkblue}
\usepackage[all]{xy}
\usepackage[bottom]{footmisc}
\usepackage{tikz}
\usetikzlibrary{trees, decorations, decorations.markings, shapes, arrows, matrix, calc, fit, intersections, patterns, angles, cd}
\usepackage[external]{forest}
\graphicspath{{figures/}{figures/rectangulations/}}
\makeatletter\def\input@path{{figures/}}\makeatother
\usepackage{caption}
\usepackage[noabbrev,capitalise]{cleveref}
\usepackage[export]{adjustbox}
\usepackage{ulem}
\usepackage{picins/picins}

\newtheorem{theorem}{Theorem}

\newtheorem{proposition}[theorem]{Proposition}
\newtheorem{lemma}[theorem]{Lemma}

\newtheorem*{theorem*}{Theorem}

\theoremstyle{definition}

\newtheorem{remark}[theorem]{Remark}

\crefname{notation}{Notation}{Notations}
\crefname{problem}{Problem}{Problems}

\newcommand{\R}{\mathbb{R}} 
\renewcommand{\c}[1]{\mathcal{#1}} 
\renewcommand{\b}[1]{{\boldsymbol{#1}}} 
\newcommand{\f}[1]{\mathfrak{#1}} 

\newcommand{\set}[2]{\left\{ #1 \;\middle|\; #2 \right\}} 
\newcommand{\bigset}[2]{\big\{ #1 \;\big|\; #2 \big\}} 
\newcommand{\Bigset}[2]{\Big\{ #1 \;\Big|\; #2 \Big\}} 
\newcommand{\ssm}{\smallsetminus} 
\newcommand{\eqdef}{\mbox{\,\raisebox{0.2ex}{\scriptsize\ensuremath{\mathrm:}}\ensuremath{=}\,}} 
\renewcommand{\implies}{\Rightarrow} 

\DeclareMathOperator{\conv}{conv} 
\DeclareMathOperator{\inv}{inv} 

\newcommand{\ie}{\textit{i.e.}~} 
\definecolor{darkblue}{rgb}{0,0,0.7} 
\definecolor{green}{RGB}{57,181,74} 
\definecolor{violet}{RGB}{147,39,143} 
\newcommand{\darkblue}{\color{darkblue}} 
\newcommand{\defn}[1]{\textsl{\darkblue #1}} 
\newcommand{\OEIS}[1]{{\rm \href{http://oeis.org/#1}{\cite[\texttt{#1}]{OEIS}}}}

\usepackage{todonotes}

\newcommand{\Vincent}[1]{\todo[inline, size=\normalsize, color=blue!30]{\rm #1 \\ \hfill --- V.}}

\newcommand{\meet}{\wedge} 
\newcommand{\join}{\vee} 
\newcommandx{\projDown}[1][1={}]{\smash{\pi_\downarrow^{#1}}} 
\newcommandx{\projUp}[1][1={}]{\smash{\pi^\uparrow_{#1}}} 

\newcommandx{\Fan}[1][1=D]{\mathcal{F}_{#1}} 
\newcommand{\polytope}[1]{\mathds{#1}} 

\newcommand{\Perm}{\polytope{P}} 
\newcommand{\Asso}{\polytope{A}} 
\newcommand{\WRP}{\polytope{WR}} 
\newcommand{\SRP}{\polytope{SR}} 
\newcommand{\SP}{\polytope{SP}}

\newcommand{\horizontalPattern}{\smash{\raisebox{-.15cm}{\includegraphics[scale=.5]{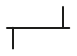}}}} 
\newcommand{\verticalPattern}{\smash{\raisebox{-.25cm}{\includegraphics[scale=.5]{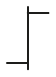}}}} 
\newcommand{\tri}{\lhd}
\newcommand{\btri}{\LHD}

\makeatletter
\newcommand{\oset}[3][0ex]{%
  \mathrel{\mathop{#3}\limits^{
    \vbox to#1{\kern-2\ex@
    \hbox{$\scriptstyle#2$}\vss}}}}
\newcommand{\uset}[3][0ex]{%
  \mathrel{\mathop{#3}\limits_{
    \vbox to#1{\kern-7\ex@
    \hbox{$\scriptstyle#2$}\vss}}}}
\makeatother
\newcommand{\loday}[1]{\smash{\overset{\frown}{#1}}}
\newcommand{\antiloday}[1]{\smash{\overset{\smile}{#1}}}
\newcommand{\upArc}[1]{\smash{\raisebox{.05cm}{$\uset[0ex]{#1}{\frown}$}}}
\newcommand{\downArc}[1]{\smash{\raisebox{-.05cm}{$\oset[.2ex]{#1}{\smile}$}}}
\newcommand{\yin}[1]{\smash{\overset{\sim}{#1}}}
\newcommand{\yang}[1]{\smash{\overset{\backsim}{#1}}}
\newcommand{\yinArc}[2]{\smash{\raisebox{.05cm}{$\uset[0ex]{#1}{\frown}$}\hspace{-.3ex}\raisebox{-.03cm}{$\oset[.2ex]{#2}{\smile}$}}}
\newcommand{\yangArc}[2]{\smash{\raisebox{-.03cm}{$\oset[.2ex]{#1}{\smile}$}\hspace{-.3ex}\raisebox{.05cm}{$\uset[0ex]{#2}{\frown}$}}}
\newcommand{\weakeq}{\asymp}
\newcommand{\strongeq}{\mathbin{\smash{\begin{smallmatrix} \backsim \\[-.3cm] \sim \end{smallmatrix}}}}

\makeatletter
\def\l@part{\@tocline{1}{8pt}{0pc}{}{}}
\def\l@section{\@tocline{1}{4pt}{0pc}{}{}}
\makeatother
\let\oldtocpart=\tocpart
\renewcommand{\tocpart}[2]{\sc\large\oldtocpart{#1}{#2}}
\let\oldtocsection=\tocsection
\renewcommand{\tocsection}[2]{\bf\oldtocsection{#1}{#2}}
\let\oldtocsubsubsection=\tocsubsubsection
\renewcommand{\tocsubsubsection}[2]{\quad\oldtocsubsubsection{#1}{#2}}

\title{Rectangulotopes}

\thanks{VP was partially supported by the Spanish grant PID2022-137283NB-C21 of MCIN/AEI/10.13039/501100011033 / FEDER, UE, by Departament de Recerca i Universitats de la Generalitat de Catalunya (2021 SGR 00697), by the French grant CHARMS (ANR-19-CE40-0017), and by the French--Austrian project PAGCAP (ANR-21-CE48-0020 \& FWF I 5788).}

\author{Jean Cardinal}
\address{Université libre de Bruxelles (ULB)}
\email{jean.cardinal@ulb.be}
\urladdr{\url{https://jean.cardinal.web.ulb.be}}

\author{Vincent Pilaud}
\address{Universitat de Barcelona (UB)}
\email{vincent.pilaud@ub.edu}
\urladdr{\url{https://www.ub.edu/comb/vincentpilaud/}}

\begin{document}

\begin{abstract}
  Rectangulations are decompositions of a square into finitely many axis-aligned rectangles.
  We describe realizations of $(n-1)$-dimensional polytopes associated with two combinatorial families of rectangulations composed of $n$ rectangles.
  They are defined as quotientopes of natural lattice congruences on the weak Bruhat order on permutations in $\f{S}_n$, and their skeleta are flip graphs on rectangulations.
  We give simple vertex and facet descriptions of these polytopes, in particular elementary formulas for computing the coordinates of the vertex corresponding to each rectangulation, in the spirit of J.-L.~Loday's realization of the associahedron.
\end{abstract}

\maketitle

\tableofcontents


\section{Introduction}
\label{sec:intro}


\subsection{Polytopes of permutations, triangulations, and rectangulations}
\label{subsec:permTrianRect}

The encoding of combinatorial objects in the form of a polyhedral structure is a recurrent theme in geometric and algebraic combinatorics.
An elementary yet striking example of this idea is the \defn{permutahedron} $\Perm (n)$.
It is the $(n-1)$-dimensional polytope defined either as the convex hull of the points $\sum_{i\in [n]} \sigma^{-1}_i\cdot \b{e}_i$ for all permutations $\sigma \in \f{S}_n$, or as the intersection of the hyperplane of~$\R^n$ given by~$\sum_{i \in [n]} x_i = \smash{\binom{n+1}{2}}$ with the halfspaces given by~$\sum_{i \in I} x_i \ge \smash{\binom{|I|+1}{2}}$ for all nonempty proper subsets~$\varnothing \ne I \subsetneq [n]$.
Its skeleton is the Cayley graph of the symmetric group for the generators consisting of adjacent transpositions.
See \cref{fig:quotientopes}\,(left) for an illustration when~$n = 4$.
The many generalizations and extensions of the permutahedra gave rise to a flourishing theory of \defn{deformed permutahedra}~\cite{Postnikov,MR2520477,MR4064768,MR4651496} (also called \defn{generalized permutahedra}, or \defn{polymatroids}).

Among those, the \defn{associahedron} is a classical and ubiquitous polytope, first defined by D.~Tama\-ri~\cite{T51} and by J.~Stasheff~\cite{S63} in a topological context.
It has since then been identified as a fundamental object in many other areas of mathematics, including operads, cluster algebras, combinatorial Hopf algebras, and physics (see for instance the recent survey from V. Pilaud, F. Santos and G. Ziegler~\cite{PilaudSantosZiegler} and the references therein).
It was first thought of as a purely combinatorial object, before various families of geometric realizations were shown to exist~\cite{MR1022776,MR1941227,MR2108555,MR3437894,MR2321739}.
The face lattice of the $(n-1)$-dimensional associahedron is the reverse inclusion poset of non-intersecting diagonals of a convex $(n+3)$-gon.
In particular, its skeleton is the \defn{flip graph} on \defn{triangulations} of the $(n+3)$-gon, or equivalently -- via a standard Catalan bijection -- the \defn{rotation graph} on \defn{binary trees} with $n$ internal nodes, the structure of which has been the subject of many investigations, with applications in computer science~\cite{MR928904,MR3197650}.
In~2004, J.-L.~Loday published the following elegant description of the associahedron~\cite{MR2108555}, giving a recipe for computing the coordinates of a vertex corresponding to a given binary tree.
We note that the inequality description of the same polytope was actually provided in~1993 by S.~Shnider and S.~Sternberg~\cite{ShniderSternberg}, but Loday's vertex description largely popularized this realization.
See \cref{fig:quotientopes}\,(right) for an illustration when~$n = 4$.

\begin{theorem}
  \label{thm:loday}
    The associahedron $\Asso (n)$ is realized by:
  \begin{enumerate}[(i)]
  \item the convex hull of the points
    \[
    \sum_{i\in [n]} \ell^T_i\cdot r^T_i \cdot \b{e}_i
    \]
    for all binary trees~$T$ with $n$ internal nodes, where $\ell^T_i$ and $r^T_i$ denote the number of leaves in the left and right subtrees of~$i$ in $T$ (labeled in inorder), see~\cite{MR2108555},
  \item the intersection of the hyperplane defined by $\sum_{i \in [n]} x_i = \binom{n+1}{2}$ with the halfspaces defined by
    \[
    \sum_{i \in I} x_i \le \#\set{J \text{ interval of } [n]}{I \cap J \ne \varnothing}
    \]
    for all intervals~$\varnothing \ne I \subsetneq [n]$, see~\cite{ShniderSternberg}.
  \end{enumerate}
\end{theorem}

A natural object associated with the associahedron is the \defn{Tamari lattice}, whose cover graph is isomorphic to the skeleton of the associahedron~\cite{MR0146227,MR3235205}.
The Tamari lattice is known to be the quotient of the weak Bruhat order by a \defn{lattice congruence} called the \defn{sylvester congruence}~\cite{MR1654173,MR2142078}.
Answering a question of N.~Reading~\cite{MR2142177}, V.~Pilaud and F.~Santos~\cite{MR3964495} proved that with \emph{every} lattice congruence of the weak Bruhat order, one can associate a polytope whose skeleton is the cover graph of the lattice quotient, and more precisely whose normal fan is the \defn{quotient fan} defined by gluing the cones of the \defn{braid fan} (the type $A$ Coxeter arrangement) that belong to the same congruence class.
Those polytopes are deformed permutahedra that they called \defn{quotientopes}.
The associahedron is therefore the quotientope of the sylvester congruence of the weak Bruhat order whose classes are in bijection with triangulations.

The goal of this paper is to provide explicit geometric realizations of the quotientopes for two particular congruences of the weak Bruhat order whose classes are in bijection with equivalence classes of rectangulations, defined as decompositions of a square into rectangles.

\begin{figure}
	\capstart
	\centerline{\includegraphics[scale=.75]{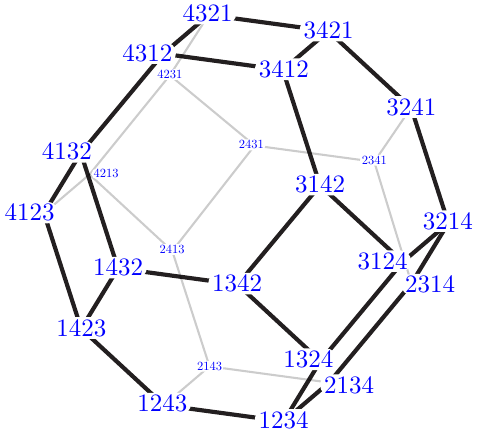} \quad \raisebox{-.5cm}{\includegraphics[scale=.75]{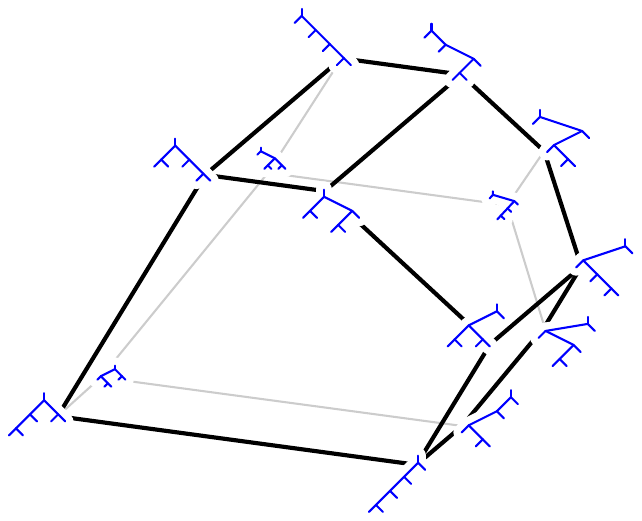}}}
	\vspace{-.2cm}
	\caption{The permutahedron~$\Perm(4)$ (left) and the associahedron~$\Asso(4)$ (right). \cite[Fig.~1]{MR3964495} \&~\cite[Fig.~5]{MR4584712}}
	\label{fig:quotientopes}
	\vspace{-.4cm}
\end{figure}


\subsection{Weak and strong rectangulations}

\enlargethispage{.3cm}
The combinatorics of rectangulations has been studied for nearly two decades~\cite{MR2233287,MR2763051,MR2871762,MR2864445,MR3084577,MR3192492,MR3878132,MR4598046,MR4667583,AB24,ACFF24}.
In fact, several combinatorial ideas can even be traced back to the problem of ``squaring the square'' studied since the 1930s~\cite{A30,Abe32,MR0000470,MR0003040}. Also, some results were initially motivated by applications to the design of \defn{floorplans} for very large scale integrated circuits, and published in electrical engineering journals, see for instance Z. C. Shen and C. C. N. Chu~\cite{SC03}, and R. Fujimaki, Y. Inoue, and T. Takahashi~\cite{FT07,TF08,ITF09,FIT09}.

We define a \defn{rectangulation} of size $n$ as a decomposition of the square into $n$ axis-parallel rectangles with disjoint interiors.
The \defn{segments} of a rectangulation are the inclusionwise maximal line segments composed of edges of the rectangles, excluding the edges of the decomposed square.
We suppose throughout that the rectangulations are \defn{generic}, in the sense that no four rectangles have a common vertex.

In order to focus on the combinatorial structure of a rectangulation, we recall two equivalence relations between rectangulations (see for instance~\cite{ACFF24} and references therein).
We say that a rectangle~$r$ is \defn{above} another rectangle~$s$ (and~$s$ is \defn{below}~$r$) if there is a sequence~${r = r_0, \dots, r_k = s}$ of rectangles such that the bottom edge of~$r_{i-1}$ lies in the same segment as the top edge of~$r_i$ for all~$i \in [k]$.
We define \defn{on the left of} and \defn{on the right of} similarly.
Two rectangulations are said to be \defn{weakly equivalent} if there exists a bijection between their rectangles that preserves the \mbox{above--below} and left--right relations between the rectangles.
Weak equivalence classes of rectangulations will be referred to as \defn{weak rectangulations}.
(Note that this is a slight, yet convenient abuse of terminology, as weak rectangulations are really sets of rectangulations.)
On the other hand, two rectangulations are said to be \defn{strongly equivalent} when there exists a bijection between their rectangles that not only preserves the above--below and left--right relations, but also the adjacency relation between rectangles.
We will naturally refer to strong equivalence classes of rectangulations as \defn{strong rectangulations}.

In order to make sense of these definitions, it is useful to consider \defn{wall slides} in a rectangulation: the local changes consisting of shifting a horizontal segment vertically, or a vertical segment horizontally, while extending or shortening the incident segments accordingly.
Performing a wall slide in a rectangulation $R$ leads to a rectangulation that is always weakly equivalent to $R$.
However, if the wall slide changes the adjacency relation between the rectangles, then the resulting rectangulation is not strongly equivalent to $R$ anymore.
A \defn{diagonal rectangulation} is a rectangulation in which every rectangle intersects the top--left to bottom--right diagonal of the square.
Examples of rectangulations are given in \cref{fig:rectequiv}.
It is simple to check that every weak rectangulation has a diagonal representative, hence weak rectangulations can also be thought of as diagonal rectangulations.

\begin{figure}
  \includegraphics[width=\textwidth]{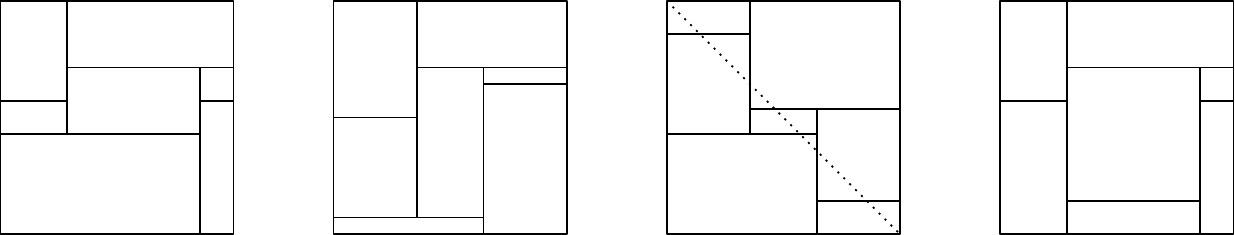}
	\caption{Four rectangulations. The rightmost rectangulation is not weakly equivalent to any other. The leftmost three are weakly equivalent, but only the first two are strongly equivalent. Only the third one is diagonal.}
	\label{fig:rectequiv}
\end{figure}


\subsection{Weak and strong rectangulotopes}

Weak and strong rectangulations of size $n$ have been shown to define congruences of the weak Bruhat order on $\f{S}_n$.

The \defn{weak rectangulation congruence}~$\weakeq$ (or \defn{Baxter congruence}) was first explicitly studied by S.~Law and N.~Reading~\cite{MR2871762} and revisited in~\cite[Thm.~1.1 \& Exm.~4.10]{Reading-arcDiagrams} in terms of arc diagrams.
Its classes are in bijection with weak rectangulations, but also with \defn{Baxter permutations}~\cite{MR0491652,MR0555815}, with \defn{twin binary trees}~\cite{MR1417289,MR2914637}, and with many other combinatorial families~\cite{MR2763051}. 
The corresponding quotientopes will be referred to as the \defn{weak rectangulotopes} and denoted by $\WRP(n)$.
We note that weak rectangulotopes were already constructed as Minkowski sums of two opposite associahedra in~\cite{MR2871762}, even before the general constructions of quotientopes of~\cite{MR3964495,MR4584712}.
For~$n = 4$, the weak rectangulation lattice is represented in \cref{fig:weakRectangulationLattice} and the weak rectangulotope is represented in \cref{fig:weakRectangulotope}.

The \defn{strong rectangulation congruence}~$\strongeq$ was studied by N.~Reading~\cite{MR2864445}, revisited~in~\cite[Thm.~1.2 \& Exm.~4.11]{Reading-arcDiagrams} in terms of arc diagrams, and studied more recently by E.~Meehan~\cite{MR3697823}, and A.~Asinowski, J.~Cardinal, S.~Felsner, and É.~Fusy~\cite{ACFF24}.
Its classes are in bijection with strong rectangulations.
The corresponding quotientopes will be referred to as the \defn{strong rectangulotopes} and denoted by~$\SRP(n)$.
They can be obtained from the constructions of~\cite{MR3964495,MR4584712}.
For~$n = 4$, the strong rectangulation lattice is represented in \cref{fig:strongRectangulationLattice} and the strong rectangulotope is represented in~\cref{fig:strongRectangulotope}.
Note that, contrarily to what could suggest this $n = 4$ case, the strong rectangulation lattice is not isomorphic to the weak order, and the strong rectangulotope is not combinatorially equivalent to the permutahedron as soon as~$n > 4$.

\begin{figure}
	\centerline{\includegraphics[scale=1.05]{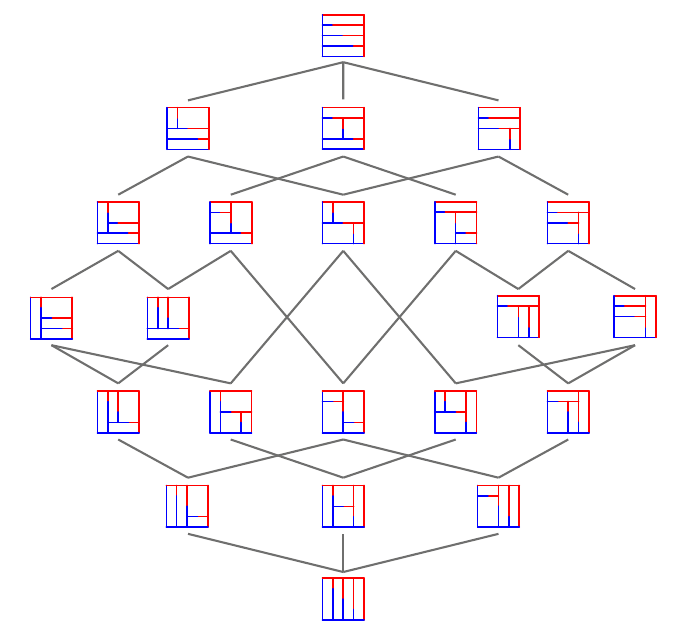}}
	\caption{The weak rectangulation lattice.}
	\label{fig:weakRectangulationLattice}
\end{figure}
\begin{figure}
	\centerline{\includegraphics[scale=1.05]{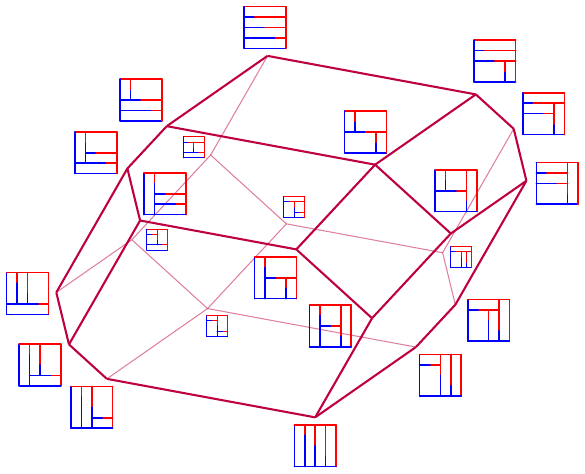}}
	\caption{The weak rectangulotope~$\WRP(4)$.}
        \label{fig:weakRectangulotope}
\end{figure}
\begin{figure}
	\centerline{\includegraphics[scale=1.05]{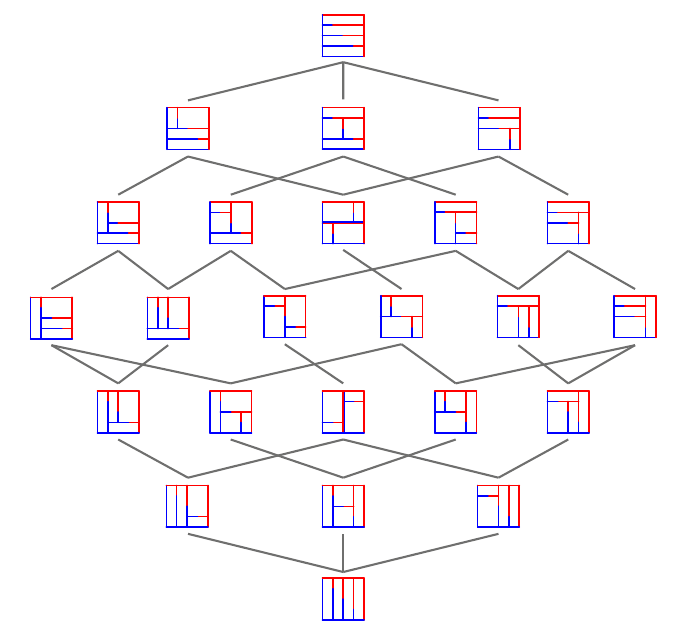}}
	\caption{The strong rectangulation lattice.}
	\label{fig:strongRectangulationLattice}
\end{figure}
\begin{figure}
	\centerline{\includegraphics[scale=1.05]{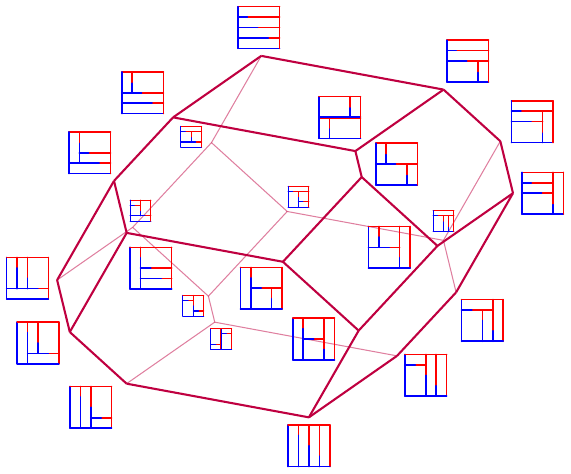}}
	\caption{The strong rectangulotope~$\SRP(4)$.}
        \label{fig:strongRectangulotope}
\end{figure}

An important motivation for studying these quotientopes is the informative structure of their skeleta, which are \defn{flip graphs} on the rectangulations.
A \defn{flip} is a local, reversible change in the structure of the rectangulation.
Flips in rectangulations come in three flavors:
\begin{itemize}
\item \defn{simple flips} replace a vertical segment incident to exactly two rectangles by a horizontal segment, and vice-versa,
\item \defn{pivoting flips} transform a pair of adjacent rectangles whose union is L-shaped, changing a left--right pair into an above--below pair, and vice-versa,
\item \defn{wall slide flips} are wall slides that modify the adjacency of the rectangles, removing exactly one adjacency and adding exactly one. These are only relevant in strong rectangulations.
\end{itemize}
The various types of flips are illustrated in \cref{fig:flips}.
Those flips define the flip graph on strong rectangulations.
We refer to \cite{MR2871762,MR3878132,MR3697823} for a detailed description of the flip graphs on weak rectangulations.
Just like associahedra encode the rotation graphs on binary trees, the rectangulotopes directly yield flip graphs on rectangulations.
Namely, the skeleton of $\WRP(n)$ (resp.~of $\SRP(n)$) is isomorphic to the flip graph on weak (resp.~strong) rectangulations of size~$n$.

\begin{figure}
  \begin{center}
  \begin{subfigure}{0.3\textwidth}
    \begin{center}
    \includegraphics[page=2, scale=.5]{flipGraph.pdf}
    \caption{Simple flips.}
    \end{center}
    \end{subfigure}
  \begin{subfigure}{0.3\textwidth}
    \begin{center}
    \includegraphics[page=1, scale=.5]{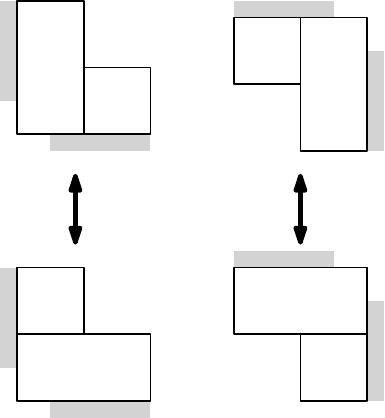}
    \caption{Pivoting flips.}
    \end{center}
  \end{subfigure}
    \begin{subfigure}{0.3\textwidth}
    \begin{center}
    \includegraphics[page=3, scale=.5]{flipGraph.pdf}
    \caption{\label{fig:wallSlides}Wall slide flips.}
    \end{center}
    \end{subfigure}
    \end{center}
	\caption{The three types of flips between strong rectangulations. The flips are possible only if the grey areas do not intersect any segment of the rectangulation. \cite[Fig.~19]{ACFF24}}
	\label{fig:flips}
\end{figure}

\begin{figure}
	\centerline{\includegraphics[width=.4\textwidth]{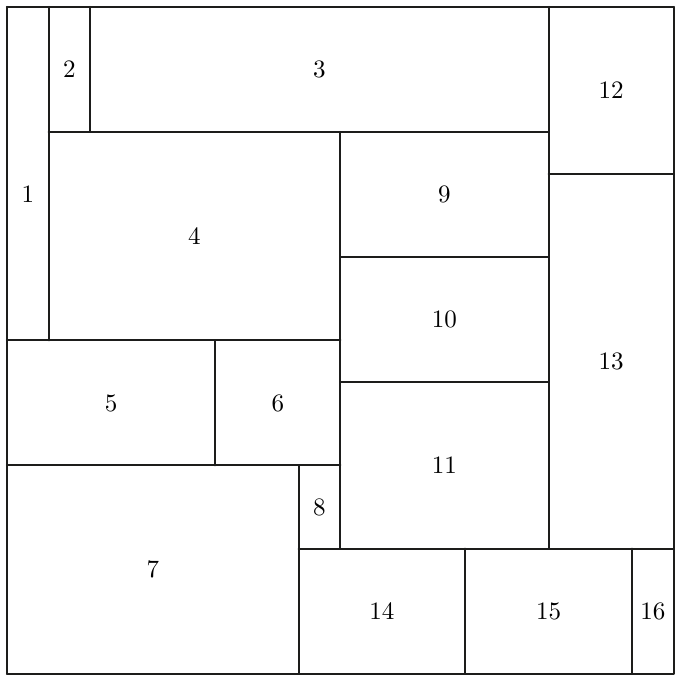} \qquad \includegraphics[width=.4\textwidth]{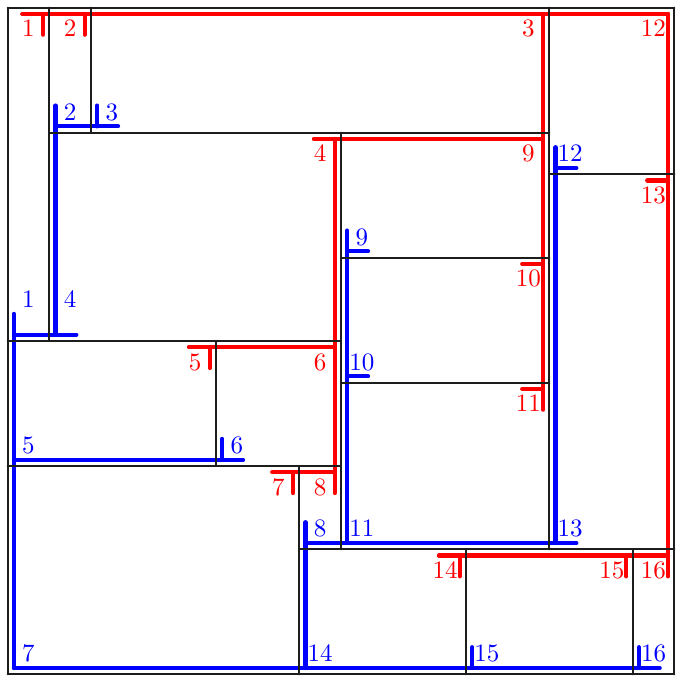}}
	\caption{A rectangulation $R$ (left) and its pair of binary trees $S(R)$ and $T(R)$ (right). Example from \cite{ACFF24}.}
        \label{fig:strongRectangulation}
\end{figure}

\begin{figure}
	\centerline{\includegraphics[width=.4\textwidth]{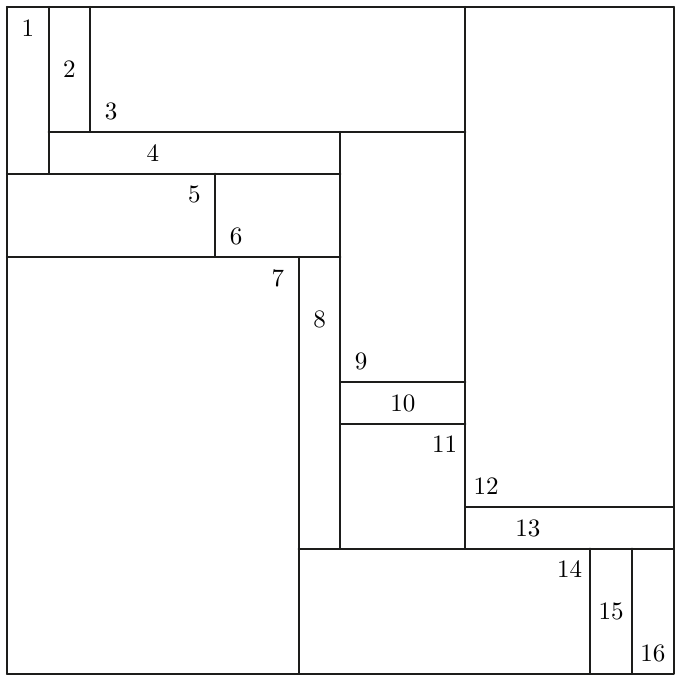} \qquad \includegraphics[width=.4\textwidth]{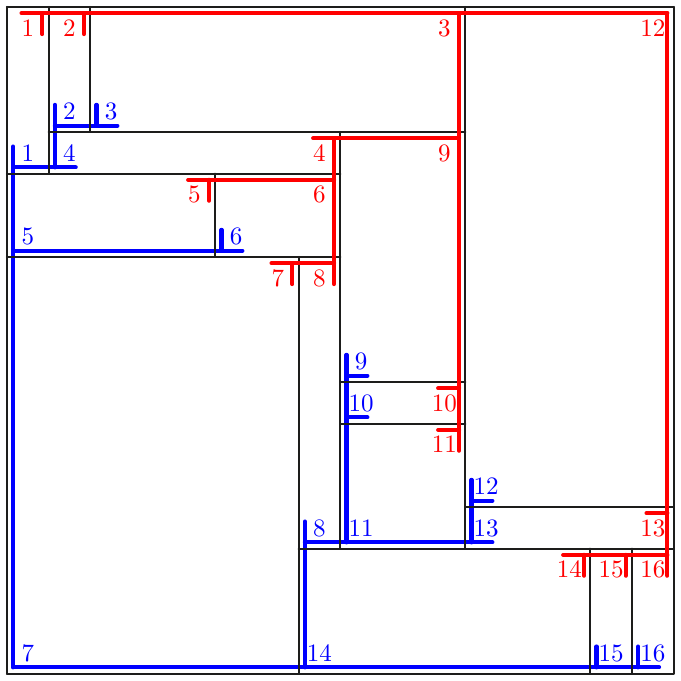}}
	\caption{A diagonal rectangulation (left) and its pair of twin binary trees (right). Example from \cite{ACFF24}.}
	\label{fig:weakRectangulation}
\end{figure}

Rectangulotopes enrich the family of natural quotientopes, alongside associahedra, permutreehedra~\cite{MR3856522}, certain brick polytopes~\cite{PilaudSantos-brickPolytopes, PilaudStump-brickPolytopes, Pilaud-brickAlgebra}, and certain graphical zonotopes~\cite{Pilaud-brickAlgebra, Pilaud-acyclicReorientationLattices}, providing instructive examples of application of the theory of lattice congruences to concrete combinatorial objects.
Our main results are elementary vertex and facet descriptions of both families of quotientopes.
For vertices, we describe the coordinates of the vertex corresponding to a weak or strong rectangulation with simple formulas in the spirit of J.-L.~Loday's realization in~\cref{thm:loday}\,(i).
As mentioned in~\cite{PilaudSantosZiegler}, there is no known such formula for arbitrary quotientopes, which makes rectangulotopes join the restricted club of Loday type quotientopes, together with associahedra and permutreehedra~\cite{MR3856522}.
For facets, we show that both weak and strong rectangulotopes have $2^n-2$ facets (the maximal possible number of facets of a deformed permutahedron), and provide simple formulas for the right hand side of the inequality corresponding to a given proper nonempty subset of~$[n]$.
We now give the necessary definitions for precisely stating these vertex and facet descriptions of the weak and strong rectangulotopes.


\subsection{Source and target trees of a rectangulation}
\label{subsec:sourceTargetTrees}

Fix a rectangulation~$R$ of size~$n$, and consider the directed graph~$D(R)$ whose vertices are all vertices of~$R$, and whose edges are obtained as follows.
For each rectangle~$r$ of~$R$, we include in~$D(R)$ four edges joining the vertices of~$r$, where the horizontal edges are oriented from left to right, and the vertical edges are oriented from bottom to top.
Hence, in the rectangle~$r$, its bottom--left corner is its source~$s(r)$, while its top--right corner is its target~$t(r)$.
Note that since~$R$ is generic, the sources~$\set{s(r)}{r \in R}$ and the targets~$\set{t(r)}{r \in R}$ are disjoint.
The \defn{source tree}~$S(R)$ (resp.~\defn{target tree}~$T(R)$) is the subgraph of~$D(R)$ induced by the sources~$\set{s(r)}{r \in R}$ (resp.~by the targets~$\set{t(r)}{r \in R}$).
It is not difficult to check that the source tree~$S(R)$ (resp.~target tree~$T(R)$) is a binary tree, rooted at the bottom--left (resp.~top--right) corner of~$R$, and oriented from (resp.~towards) its root.
Observe also that $S(R)$ and~$T(R)$ only depend on the weak class of the rectangulation~$R$.

We complete each node of~$S(R)$ and~$T(R)$ with a vertical (resp.~horizontal) leaf if it has no vertical (resp.~horizontal) child.
An example is given in~\cref{fig:strongRectangulation}.
If the rectangulation $R$ is diagonal, we retrieve the well-studied pair of twin binary trees~\cite{MR1417289,MR2914637}, see~\cref{fig:weakRectangulation}.

Recall that the \defn{inorder labeling} of a binary tree~$T$ is the labeling of the nodes of~$T$ such that the label of each node~$t$ of~$T$ is larger than all labels in the left subtree of~$t$ and smaller than all labels in the right subtree of~$t$.
For any rectangle~$r$ of~$R$, the inorder label of the source~$s(r)$ in~$S(R)$ coincides with the inorder label of the target~$t(r)$ in~$T(R)$.
The inorder can also be retrieved from the above--below and left--right relationships defined before, as the transitive closure of the union of the below and right partial orders.
This enables to unambiguously label the rectangles of~$R$ by the inorder.
The resulting labeling coincides with the NW--SE labeling of~\cite{ACFF24}.
From now on, the labels of the rectangles in a rectangulation are the inorder labels, and the two trees~$T(R)$ and $S(R)$ are defined on the vertex set $[n]$ accordingly.
For each vertex in $T(R)$ (or in~$S(R)$), we refer to its subtree reached by a horizontal (resp.~vertical) edge as its \defn{horizontal} (resp.~\defn{vertical})~\defn{subtree}.


\subsection{Realizations of weak rectangulotopes}
\label{subsec:weakRectangulotopes}

Our first results are simple vertex and facet descriptions of the weak rectangulotopes~$\WRP(n)$.
For the vertices, we provide a concise formula for the coordinates, that consists of applying J.-L.~Loday's formula on each of the source and target trees of the rectangulation.
For the facets, we combine the right hand sides of two opposite associahedra to obtain those of the weak rectangulotopes.
These two descriptions materialize the result of Law and Reading that the weak rectangulotopes are Minkowski sums of two opposite associahedra~\cite{MR2871762}.

\begin{theorem}
  \label{thm:weakRectangulotope}
The weak rectangulotope $\WRP (n)$ is realized by the polytope equivalently described as:
  \begin{enumerate}[(i)]
  \item the convex hull of the points
    \[
    \sum_{i\in [n]} (\loday{w}^R_i - \antiloday{w}^R_i)\cdot \b{e}_i
    \]
    for all weak rectangulations $R$ of size~$n$, with
  \[
    \loday{w}^R_i \eqdef h^T_i\cdot v^T_i
    \qquad\text{and}\qquad
    \antiloday{w}^R_i \eqdef h^S_i\cdot v^S_i,
  \]
  where~$S$ and~$T$ are the source and target trees of the rectangulation~$R$, and $h^T_i$ and $v^T_i$ denote the number of leaves in the horizontal and vertical subtrees of $i$ in~$T$,
\item the intersection of the hyperplane defined by~$\sum_{i \in [n]} x_i = 0$ with the halfspaces defined by
  \[
  \sum_{i \in X} x_i \le \#\set{I \text{ interval of } [n]}{I \not\subseteq X \text{ and } I \not\subseteq [n] \ssm X}
  \]
  for all subsets~$\varnothing \ne X \subsetneq [n]$.
  \end{enumerate}
\end{theorem}

\begin{table}
	\centerline{
	\begin{tabular}{cccccc}
		\includegraphics[scale=2]{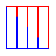} &
		\includegraphics[scale=2]{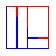} &
		\includegraphics[scale=2]{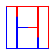} &
		\includegraphics[scale=2]{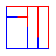} &
		\includegraphics[scale=2]{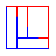} &
		\includegraphics[scale=2]{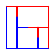}
		\\[-.1cm]
		$(-3, -1, 1, 3)$ &
		$(-3, -1, 5, -1)$ &
		$(-3, 3, -3, 3)$ &
		$(1, -5, 1, 3)$ &
		$(-3, 0, 5, -2)$ &
		$(-3, 5, -3, 1)$
		\\[.2cm]
		\includegraphics[scale=2]{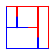} &
		\includegraphics[scale=2]{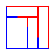} &
		\includegraphics[scale=2]{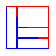} &
		\includegraphics[scale=2]{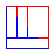} &
		\includegraphics[scale=2]{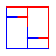} &
		\includegraphics[scale=2]{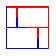}
		\\[-.1cm]
		$(-1, 3, -5, 3)$ &
		$(2, -5, 0, 3)$ &
		$(-3, 5, 0, -2)$ &
		$(-2, 0, 5, -3)$ &
		$(1, -5, 5, -1)$ &
		$(-1, 5, -5, 1)$
		\\[.2cm]
		\includegraphics[scale=2]{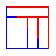} &
		\includegraphics[scale=2]{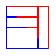} &
		\includegraphics[scale=2]{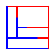} &
		\includegraphics[scale=2]{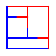} &
		\includegraphics[scale=2]{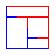} &
		\includegraphics[scale=2]{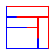}
		\\[-.1cm]
		$(3, -5, 0, 2)$ &
		$(2, 0, -5, 3)$ &
		$(-2, 5, 0, -3)$ &
		$(1, -3, 5, -3)$ &
		$(3, -5, 3, -1)$ &
		$(3, 0, -5, 2)$
		\\[.2cm]
		&		
		\includegraphics[scale=2]{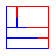} &
		\includegraphics[scale=2]{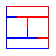} &
		\includegraphics[scale=2]{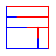} &
		\includegraphics[scale=2]{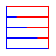} &
		\\[-.1cm]
		&
		$(-1, 5, -1, -3)$ &
		$(3, -3, 3, -3)$ &
		$(3, 1, -5, 1)$ &
		$(3, 1, -1, -3)$ &
	\end{tabular}
	}
	\caption{The vertices of the weak rectangulotope~$\WRP(4)$ of \cref{fig:weakRectangulotope}.}
	\label{tab:verticesWeakRectangulotope}
\end{table}

\begin{table}
	\centerline{
	\begin{tabular}{cccccc}
		\includegraphics[scale=2]{rectangulation1} &
		\includegraphics[scale=2]{rectangulation2} &
		\includegraphics[scale=2]{rectangulation3} &
		\includegraphics[scale=2]{rectangulation4} &
		\includegraphics[scale=2]{rectangulation5} &
		\includegraphics[scale=2]{rectangulation6}
		\\[-.1cm]
		$(-6, -2, 2, 6)$ &
		$(-6, -2, 9, -1)$ &
		$(-6, 5, -5, 6)$ &
		$(1, -9, 2, 6)$ &
		$(-6, 1, 9, -4)$ &
		$(-6, 9, -5, 2)$
		\\[.2cm]
		\includegraphics[scale=2]{rectangulation8} &
		\includegraphics[scale=2]{rectangulation9} &
		\includegraphics[scale=2]{rectangulation10} &
		\includegraphics[scale=2]{rectangulation11} &
		\includegraphics[scale=2]{rectangulation12} &
		\includegraphics[scale=2]{rectangulation13}
		\\[-.1cm]
		$(-2, 5, -9, 6)$ &
		$(4, -9, -1, 6)$ &
		$(-6, 9, 1, -4)$ &
		$(-4, 1, 9, -6)$ &
		$(2, -9, 9, -2)$ &
		$(-2, 9, -9, 2)$
		\\[.2cm]
		\includegraphics[scale=2]{rectangulation14} &
		\includegraphics[scale=2]{rectangulation15} &
		\includegraphics[scale=2]{rectangulation16} &
		\includegraphics[scale=2]{rectangulation17} &
		\includegraphics[scale=2]{rectangulation19} &
		\includegraphics[scale=2]{rectangulation20}
		\\[-.1cm]
		$(6, -9, -1, 4)$ &
		$(4, -1, -9, 6)$ &
		$(-4, 9, 1, -6)$ &
		$(2, -5, 9, -6)$ &
		$(6, -9, 5, -2)$ &
		$(6, -1, -9, 4)$
		\\[.2cm]
		\includegraphics[scale=2]{rectangulation21} &
		\includegraphics[scale=2]{rectangulation22} &
		\includegraphics[scale=2]{rectangulation23} &
		\includegraphics[scale=2]{rectangulation24} &
		\includegraphics[scale=2]{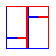} &
		\includegraphics[scale=2]{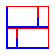}
		\\[-.1cm]
		$(-1, 9, -2, -6)$ &
		$(6, -5, 5, -6)$ &
		$(6, 2, -9, 1)$ &
		$(6, 2, -2, -6)$ &
		$(1, -9, 9, -1)$ &
		$(-1, 9, -9, 1)$
	\end{tabular}
	}
	\caption{The vertices of the strong rectangulotope~$\SRP(4)$ of \cref{fig:strongRectangulotope}.}
	\label{tab:verticesStrongRectangulotope}
\end{table}

For instance, the weak rectangulotope~$\WRP(4)$ is illustrated in \cref{fig:weakRectangulotope}.
The coordinates of the vertices of~$\WRP(4)$ corresponding to all weak rectangulations are gathered in \cref{tab:verticesWeakRectangulotope}.
The facet inequalities of~$\WRP(4)$ corresponding to each nonempty proper subset of~$[4]$ are gathered~in~\cref{tab:facetsWeakRectangulotope}.

\begin{table}
	\centerline{
	\begin{tabular}{c|cccccccccccccc}
		$J$ &
		$1$ &
		$2$ &
		$3$ &
		$4$ &
		$12$ &
		$13$ &
		$14$ &
		$23$ &
		$24$ &
		$34$ &
		$123$ &
		$124$ &
		$134$ &
		$234$
		\\
		\hline
		$f_{\weakeq}(J)$ &
		$3$ &
		$5$ &
		$5$ &
		$3$ &
		$4$ &
		$6$ &
		$5$ &
		$5$ &
		$6$ &
		$4$ &
		$3$ &
		$5$ &
		$5$ &
		$3$
	\end{tabular}
	}
	\vspace{-.2cm}
	\caption{The facets of the weak rectangulotope~$\WRP(4)$ of \cref{fig:weakRectangulotope}.}
	\label{tab:facetsWeakRectangulotope}
\end{table}

  The vertex of~$\WRP (16)$ corresponding to the weak rectangulation of \cref{fig:weakRectangulation} is
  \[
  (-3, 0, 26, 2, -9, 5, -69, -4, 17, 0, -8, 59, 2, -20, 0, 2).
  \]


\subsection{Realizations of strong rectangulotopes}
\label{subsec:strongRectangulotopes}

In order to exhibit a similar realization for the strong rectangulotopes, we need to count subsets of leaves that lie in some subtrees of both the source and target trees of a rectangulation.
Two leaves of the trees $T(R)$ and $S(R)$ are said to be \defn{common leaves} if the edges to their parents lie on the same segment of $R$.
Note that the common leaves of~$T(R)$ and~$S(R)$ only depend on the weak class of the rectangulation~$R$.

We denote by ${}_i \!\!\! \raisebox{.03cm}{\verticalPattern} \!\!\! {}^j$ the situation where the rectangles $r_i$ and $r_j$ of inorder labels $i$ and $j$ share a vertical segment containing the right edge of~$r_i$ and the left edge of~$r_j$, and the bottom edge of $r_i$ is below the top edge of~$r_j$.
Similarly, we denote by~$\;\;\; {}^i \!\!\!\!\!\!\!\! \horizontalPattern \!\!\!\!\!\!\!\! \raisebox{-.06cm}{$_j$} \;\;\;$ the situation where the rectangles $r_i$ and $r_j$ share a horizontal segment containing the bottom edge of~$r_i$ and the top edge of~$r_j$, and the right edge of~$r_i$ is on the right of the left edge of~$r_j$.
We use the Iverson bracket $\llbracket \varphi\rrbracket$, which is equal to 1 if $\varphi$ holds, and 0 otherwise, and~$\neg \, \varphi$ for the logical negation of~$\varphi$.

\begin{theorem}
\label{thm:strongRectangulotope}
The strong rectangulotope~$\SRP (n)$ is realized by the polytope equivalently described as:
  \begin{enumerate}[(i)]
  \item the convex hull of the points
  \[
  \sum_{i,j\in [n], i< j} (\yin{w}^R_{i,j} - \yang{w}^R_{i,j})\cdot (\b{e}_i - \b{e}_j),
  \]
   for all strong rectangulations $R$ of size~$n$, with
  \[
    \yin{w}^R_{i,j} \eqdef h^T_i \cdot cv^{T,S}_{i,j}\cdot h^S_ j\cdot \llbracket \; \neg \; {}_i \!\!\! \raisebox{.03cm}{\verticalPattern} \!\!\! {}^j \; \rrbracket
    \qquad\text{and}\qquad
    \yang{w}^R_{i,j} \eqdef v^S_i \cdot ch^{S,T}_{i,j}\cdot v^T_j \cdot \llbracket \; \neg \;\;\;\; {}^i \!\!\!\!\!\!\!\! \horizontalPattern \!\!\!\!\!\!\!\! \raisebox{-.06cm}{$_j$} \;\; \;\; \rrbracket,
  \]
 where:
  \begin{itemize}
  \item $S$ and $T$ are the source and target trees of the rectangulation~$R$,
  \item $h^T_i$ and $v^T_i$ denote the number of leaves in the horizontal and vertical subtrees of $i$ in~$T$,
  \item $ch^{T,S}_{i,j}$ (resp.~$cv^{T,S}_{i,j}$) denote the number of common leaves of the horizontal (resp.~vertical) subtree of $i$ in $T$ and the horizontal (resp.~vertical) subtree of $j$ in $S$,
  \end{itemize}
\item the intersection of the hyperplane defined by~$\sum_{i \in [n]} x_i = 0$ with the halfspaces defined by
  \[
  \sum_{i \in X} x_i \le \#\set{I,J}{I \not\subseteq X \text{ and } J \not\subseteq [n] \ssm X} + \#\set{I,J}{I \not\subseteq [n] \ssm X \text{ and } J \not\subseteq X}
  \]
  for all subsets~$\varnothing \ne X \subsetneq [n]$, where~$I,J$ are nonempty consecutive intervals of~$[n]$, meaning that~$\max(I) = \min(J)-1$.
  \end{enumerate}
\end{theorem}


For instance, the strong rectangulotope~$\SRP(4)$ is illustrated in \cref{fig:strongRectangulotope}.
The coordinates of the vertices of~$\SRP(4)$ corresponding to all strong rectangulations are gathered in \cref{tab:verticesStrongRectangulotope}.
The facet inequalities of~$\SRP(4)$ corresponding to each nonempty proper subset of~$[4]$ are gathered in \cref{tab:facetsStrongRectangulotope}.

\begin{table}
	\centerline{
	\begin{tabular}{c|cccccccccccccc}
		$J$ &
		$1$ &
		$2$ &
		$3$ &
		$4$ &
		$12$ &
		$13$ &
		$14$ &
		$23$ &
		$24$ &
		$34$ &
		$123$ &
		$124$ &
		$134$ &
		$234$
		\\
		\hline
		$f_{\strongeq}(J)$ &
		$6$ &
		$9$ &
		$9$ &
		$6$ &
		$8$ &
		$11$ &
		$10$ &
		$10$ &
		$11$ &
		$8$ &
		$6$ &
		$9$ &
		$9$ &
		$6$
	\end{tabular}
	}
	\vspace{-.2cm}
	\caption{The facets of the strong rectangulotope~$\SRP(4)$ of \cref{fig:strongRectangulotope}.}
	\label{tab:facetsStrongRectangulotope}
\end{table}

\pagebreak
  The vertex of~$\SRP (16)$ corresponding to the strong rectangulation of \cref{fig:strongRectangulation} is
  \[
  (-41, 5, 262, 24, -110, 50, -525, -50, 184, 3, -88, 450, 35, -221, -5, 27).
  \]

We think of the formula of \cref{thm:strongRectangulotope}\,(i) as the analogue of J.-L.-Loday's product formula of \cref{thm:loday}\,(i) for the coordinates of the vertices of the associahedron.
We note however that our formula has a quadratic number of terms.
It remains unclear to us if there is a simple linear formula to compute the coordinates of the vertex of~$\SRP(n)$ corresponding to a given strong rectangulation.


\subsection{Plan of the paper}
\label{subsec:plan}

\cref{sec:quotientopes} is dedicated to the background on quotientopes and their realizations as Minkowski sums of \defn{shard polytopes}, due to A.~Padrol, V.~Pilaud, and J.~Ritter~\cite{MR4584712}, which will be our main tool throughout. \cref{sec:weakRectangulotopes} details the case of weak rectangulations, while \cref{sec:strongRectangulotopes} deals with strong rectangulations. In both cases, we provide both the vertex and facet descriptions of these polytopes.


\subsection{Acknowledgments}

This work was initiated at the Workshop on Combinatorics, Algorithms, and Geometry held on March 4--8, 2024 in Dresden, Germany.
The authors thank Namrata and Torsten M\"utze for the organization and the other participants of the workshop for the stimulating interactions, notably on other combinatorial aspects of rectangulations.
We are grateful to two anonymous referees for many suggestions on this paper.


\section{Quotientopes}
\label{sec:quotientopes}

We now briefly recall some results on the weak Bruhat order and its quotients, both from a lattice and geometric perspective.
We refer to \cite{MR2142177, Reading-arcDiagrams, MR3645056, MR3645055, MR3964495, MR4584712} for details.


\subsection{Weak order and noncrossing arc diagrams}
\label{subsec:noncrossingArcDiagrams}

Denote by~$\f{S}_n$ the set of permutations of~$[n]$.
The \defn{inversion set} of~$\sigma \in \f{S}_n$ is~$\inv(\sigma) \eqdef \set{(\sigma_i, \sigma_j)}{1 \le i < j \le n \text{ and } \sigma_i > \sigma_j}$.
The \defn{weak Bruhat order}\footnote{The weak Bruhat order is usually just referred to as the weak order~\cite{MR1066460, MR2133266}. In this paper, we prefer to use weak Bruhat order to avoid any confusion with the weak poset defined on weak rectangulations.} is the lattice on the permutations of~$\f{S}_n$ defined by the inclusion of their inversion sets.
See \cref{fig:sylvesterCongruence}\,(left).
Note that the cover relations in the weak Bruhat order are given by the transpositions of two adjacent letters.

An \defn{arc} on~$[n]$ is a quadruple~$(a, b, A, B)$ where~$1 \le a < b \le n$ and~$A \sqcup B$ forms a partition of the interval~${]a,b[} \eqdef \{a+1, \dots, b-1\}$.
We represent an arc by an abscissa monotone curve wiggling around the horizontal axis, starting at point~$a$ and ending at point~$b$, and passing above the points of~$A$ and below the points of~$B$.
A \defn{noncrossing arc diagram} on~$[n]$ is a collection of arcs on~$[n]$ where any two arcs do not cross in their interior and have distinct left endpoints and distinct right endpoints (but the right endpoint of an arc can be the left endpoint of another arc).

\vspace{-.1cm}
\parpic(4cm,2.5cm)(10pt, 150pt)[r][b]{\includegraphics[scale=.9]{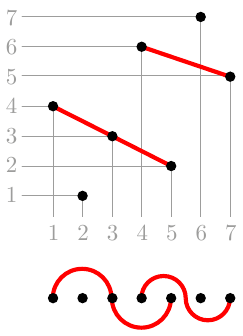}}{
In~\cite{Reading-arcDiagrams}, N.~Reading defined an elegant bijection between permutations of~$[n]$ and noncrossing arc diagrams on~$[n]$.
It sends a permutation~$\sigma$ to the noncrossing arc diagram with an arc~$(\sigma_{j+1}, \sigma_j, A_j, B_j)$ for each descent~$j \in [n-1]$ of~$\sigma$ (\ie with~$\sigma_j > \sigma_{j+1}$), where

\vspace{-.2cm}
\begin{minipage}{10cm}
\begin{align*}
A_j & \eqdef \set{\sigma_i}{1 \le i < j \text{ and } \sigma_{j+1} < \sigma_i < \sigma_j} \\
\text{and} \qquad
B_j & \eqdef \set{\sigma_k}{j+1 < k \le n \text{ and } \sigma_{j+1} < \sigma_k < \sigma_j}.
\end{align*}
\end{minipage}

\vspace{.2cm}
\noindent
As illustrated on the right with~$\sigma = 2531746$, this can also been visualized by representing the table~$(\sigma_j,j)$ of the permutation~$\sigma$, drawing the segments corresponding to the descents of~$\sigma$, and letting all points fall on the horizontal axis, allowing the segments to bend but not to cross each other nor to pass through a point.
The single arcs correspond to permutations with a single descent, that is, to join irreducible permutations of the weak Bruhat order.
In general, the noncrossing arc diagram of a permutation~$\sigma$ actually encodes the canonical join representation of~$\sigma$ in the weak Bruhat order (which was known to exist, as the weak Bruhat order is join semidistributive).
See~\cite{Reading-arcDiagrams} for details.
}


\subsection{Quotients and arc ideals}
\label{subsec:arcIdeals}

\afterpage{
\begin{figure}
	\capstart
	\centerline{\includegraphics[scale=.6]{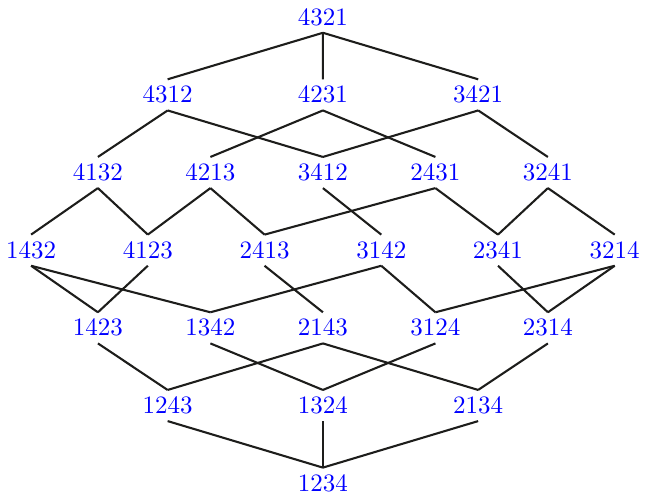} \; \includegraphics[scale=.6]{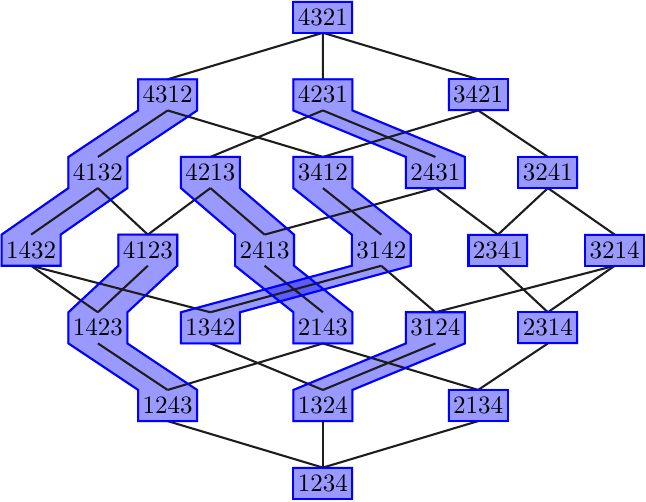} \; \includegraphics[scale=.48]{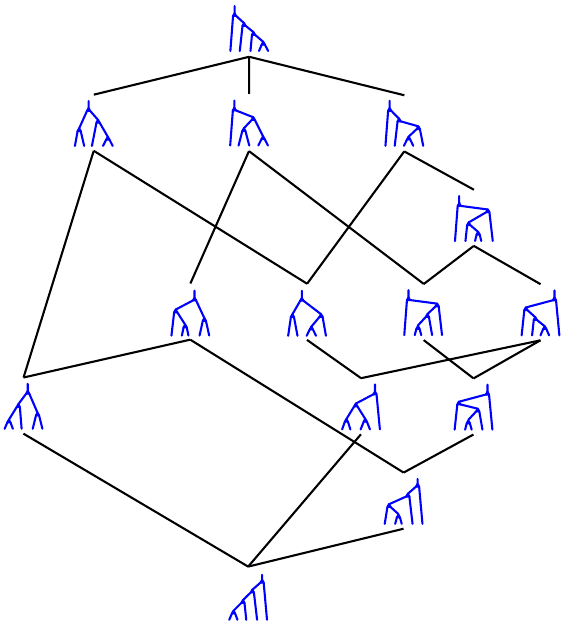}}
	\caption{The weak Bruhat order on~$\f{S}_4$ (left), the sylvester congruence~$\equiv_\textrm{sylv}$~(middle), and the Tamari lattice (right). \cite[Fig.~1 \& 2]{MR3964495}}
	\label{fig:sylvesterCongruence}
\end{figure}
}

A \defn{lattice congruence} of the weak Bruhat order is an equivalence relation~$\equiv$ on~$\f{S}_n$ that respects the meet and join operations, \ie such that $x \equiv x'$ and~$y \equiv y'$ implies $x \meet y \, \equiv \, x' \meet y'$ and~$x \join y \, \equiv \, x' \join y'$.
The \defn{lattice quotient}~$\f{S}_n/{\equiv}$ is the lattice on the congruence classes of~$\equiv$ where~$X \le Y$ if and only if there exist~$x \in X$ and~$y \in Y$ such that~$x \le y$, and~$X \meet Y$ (resp.~$X \join Y$) is the congruence class of~$x \meet y$ (resp.~$x \join y$) for any~$x \in X$~and~$y \in Y$.

For instance, the sets of linear extensions of binary trees (labeled in inorder and oriented toward their roots) are the classes of the \defn{sylvester congruence}~\cite{MR2142078}.
See \cref{fig:sylvesterCongruence}\,(middle).
The quotient of the weak Bruhat order by the sylvester congruence is the classical \defn{Tamari lattice}.
See \cref{fig:sylvesterCongruence}\,(right).

Respecting the meet and join operation is a strong condition that imposes additional structure on~$\equiv$.
In fact, all the congruence classes of~$\equiv$ are intervals of the weak Bruhat order, and the quotient~$\f{S}_n/{\equiv}$ is isomorphic (as a poset) to the subposet of the weak Bruhat order induced by permutations which are minimal in their class.
Moreover, the latter are precisely the permutations whose noncrossing arc diagrams only use arcs corresponding to join irreducible permutations which are minimal in their class, and we denote by~$\c{A}_\equiv$ this set of arcs.
In other words, the classes of~$\equiv$ are in bijection with noncrossing arc diagrams using only arcs in~$\c{A}_\equiv$.
This bijection actually translates the fact that canonical join representations behave properly under lattice quotients.

For instance, the join irreducible permutations minimal in their sylvester congruence class are precisely the \defn{up arcs}, \ie the arcs of the form~$(a, b, {]a,b[}, \varnothing)$ (similarly, we call \defn{down arcs} the arcs of the form~$(a, b, \varnothing, {]a,b[})$).
The corresponding noncrossing arc diagrams are famously known as noncrossing partitions.

\vspace{-.1cm}
\parpic(5cm,2.8cm)(10pt, 90pt)[r][b]{\includegraphics[scale=.5]{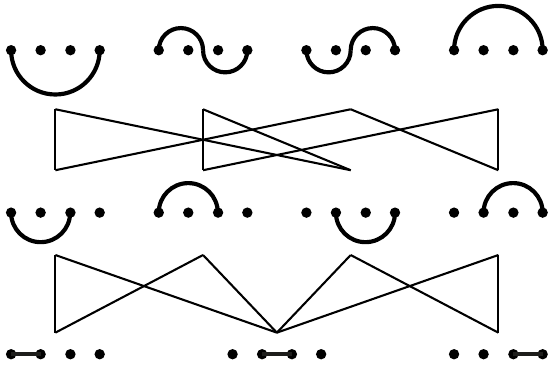}}{
An arc~$(a, b, A, B)$ is a \defn{subarc} of an arc~$(a', b', A', B')$ if \linebreak $a' \le a < b \le b'$ and~$A \subseteq A'$ while~$B \subseteq B'$.
An \defn{arc ideal} is a subset of arcs closed by subarcs.
The map~${\equiv} \mapsto \c{A}_\equiv$ is a bijection between the lattice congruences of the weak Bruhat order and the arc ideals.
In other words, the lattice of congruences of the weak Bruhat order is distributive, and its poset of join irreducibles is isomorphic to the subarc order, illustrated on the right for~$n = 4$.
}


\subsection{Quotient fans and shards}
\label{subsec:quotientFans}

\begin{figure}
	\capstart
	\centerline{\includegraphics[scale=.75]{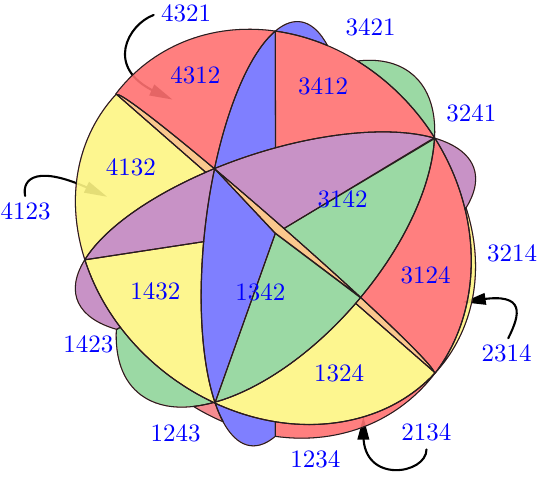} \quad \raisebox{-.35cm}{\includegraphics[scale=.75]{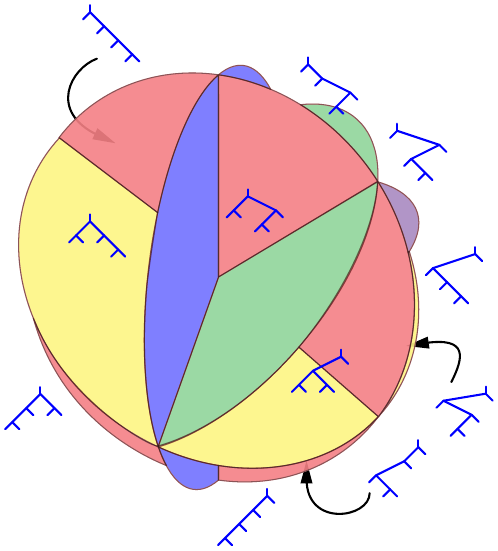}}}
	\caption{The braid fan (left) and the sylvester fan (right). \cite[Fig.~1]{MR3964495} \&~\cite[Fig.~5]{MR4584712}}
	\label{fig:fans}
\end{figure}

The \defn{braid arrangement} is the arrangement consisting of the hyper\-planes~$\set{\b{x} \in \R^n}{x_a = x_b}$ for all~$1 \le a < b \le n$.
It has a chamber~$\set{\b{x} \in \R^n}{x_{\sigma_1} \le \dots \le x_{\sigma_n}}$ for each permutation~$\sigma$ of~$[n]$. 
It has a ray for each nonempty proper subset~$\varnothing \ne J \subsetneq [n]$.
See \cref{fig:fans}\,(left).
(A ray is a dimension~$1$ cone in the essentialization of the braid arrangement obtained by slicing it with a hyperplane normal to its lineality space, which is generated by the vector~$(1, \dots, 1)$.)

The \defn{shard} of an arc~$(a, b, A, B)$ is the piece of the braid hyperplane~${x_a = x_b}$ defined by the inequalities~$x_{a'} < x_a$ for all~$a' \in A$ and~$x_b < x_{b'}$ for all~$b' \in B$.
In other words, the shards decompose each hyperplane~$x_a = x_b$ of the braid arrangement into~$2^{b-a-1}$ cones.

The \defn{quotient fan} of a congruence~$\equiv$ of the weak Bruhat order is the polyhedral fan~$\c{F}_\equiv$~where
\begin{itemize}
\item the maximal cones are obtained by glueing together the chambers of the braid arrangement corresponding to permutations in the same congruence class of~$\equiv$,
\item the union of the codimension~$1$ cones is the union of the shards of the arcs of~$\c{A}_\equiv$.
\end{itemize}
(These two descriptions are equivalent.)
By construction, the braid fan refines the quotient fan~$\c{F}_\equiv$, and the dual graph of the quotient fan~$\c{F}_\equiv$ is isomorphic to the cover graph of the quotient~$\f{S}_n/{\equiv}$.
For instance, \cref{fig:fans}\,(right) represents the \defn{sylvester fan} (the quotient fan of the sylvester congruence of \cref{fig:sylvesterCongruence}\,(middle)) for~$n = 4$, whose dual graph is the cover graph of the Tamari lattice of \cref{fig:sylvesterCongruence}\,(right).
The maximal cone corresponding to a binary tree~$T$ is given by~$\polytope{C}^T = \set{\b{x} \in \R^n}{x_i \le x_j \text{ for each edge } i \to j \text{ in } T}$.

It was shown in~\cite[Lem.~23]{MR4328906} that the ray of the braid arrangement corresponding to a nonempty proper subset~$\varnothing \ne J \subsetneq [n]$ is preserved in the quotient fan~$\c{F}_\equiv$ if and only if the arc ideal~$\c{A}_\equiv$ contains the $|J|-1$ down arcs joining two consecutive elements of~$J$ and the~$n-|J|-1$ up arcs joining two consecutive elements of~$[n] \ssm J$.
An immediate consequence is that all $2^n-2$ rays of the braid arrangement are preserved in the quotient fan~$\c{F}_\equiv$ if and only if the arc ideal~$\c{A}_\equiv$ contains all up arcs and all down arcs.


\subsection{Quotientopes and shard polytopes}
\label{subsec:quotientopes}

A \defn{quotientope} for a lattice congruence~$\equiv$ of the weak Bruhat order is a polytope whose normal fan is the quotient fan~$\c{F}_\equiv$.
In particular, the skeleton of a quotientope, oriented in the direction~$(n,n-1, \dots, 1) - (1, 2, \dots, n) = (n-1, n-3, \dots, 3-n, 1-n)$, is isomorphic to the Hasse diagram of the quotient~$\f{S}_n/{\equiv}$.
For instance, the associahedron~$\Asso(n)$ (\cref{fig:quotientopes}\,(right) when~$n = 4$) is a quotientope for the sylvester congruence (\cref{fig:sylvesterCongruence}\,(middle) when~$n = 4$), its normal fan is the sylvester fan (\cref{fig:fans}\,(right) when~$n = 4$), and its oriented skeleton is the Hasse diagram of the Tamari lattice (\cref{fig:sylvesterCongruence}\,(right) when~$n = 4$).
The existence of quotientopes was first proved in \cite{MR3964495} and later better understood in~\cite{MR4584712} using Minkowski sums of shard polytopes, the approach we will use throughout the paper.

The \defn{Minkowski sum} of two polytopes~$\polytope{P}, \polytope{Q} \subset \R^n$ is~$\polytope{P} + \polytope{Q} \eqdef \set{\b{p} + \b{q}}{\b{p} \in \polytope{P} \text{ and } \b{q} \in \polytope{Q}} \subset \R^n$.
Recall that for any vector~$\b{v} \ne \b{0}$ of~$\R^n$, the face of~$\polytope{P} + \polytope{Q}$ maximizing the scalar product with~$\b{v}$ is the Minkowski sum of the faces of~$\polytope{P}$ and~$\polytope{Q}$ maximizing the scalar product with~$\b{v}$.
In particular, if~$\b{v}$ is a generic direction, the vertex of~$\polytope{P} + \polytope{Q}$ that is extremal in direction~$\b{v}$ is just the sum of the vertices of~$\polytope{P}$ and~$\polytope{Q}$ that are extremal in direction~$\b{v}$.
Moreover, the normal fan of the Minkowski sum~$\polytope{P} + \polytope{Q}$ is the common refinement of the normal fans of~$\polytope{P}$ and~$\polytope{Q}$.

Consider an arc~$\alpha \eqdef (a, b, A, B)$ on~$[n]$.
An \defn{$\alpha$-alternating matching} is a sequence~$a \le i_1 < j_1 < i_2 < j_2 < \dots < i_q < j_q \le b$ such that~$i_p \in \{a\} \cup A$ and~$j_p \in \{b\} \cup B$ for all~$p \in [q]$.
Its \defn{characteristic vector} is~$\sum_{p \in [q]} (\b{e}_{i_p} - \b{e}_{j_p})$.
The \defn{shard polytope} of~$\alpha$ is the convex hull~$\SP(\alpha)$ of the characteristic vectors of all $\alpha$-alternating matchings.
We refer to \cref{fig:shardPolytopes} for an illustration of this definition.
It was shown in~\cite{MR4584712} that any Minkowski sum of positive scalings of the shard polytopes~$\SP(\alpha)$ for all arcs~$\alpha$ in the arc ideal~$\c{A}_\equiv$ is a quotientope for~$\equiv$.

\begin{figure}
	\capstart
	\centerline{\includegraphics[scale=1]{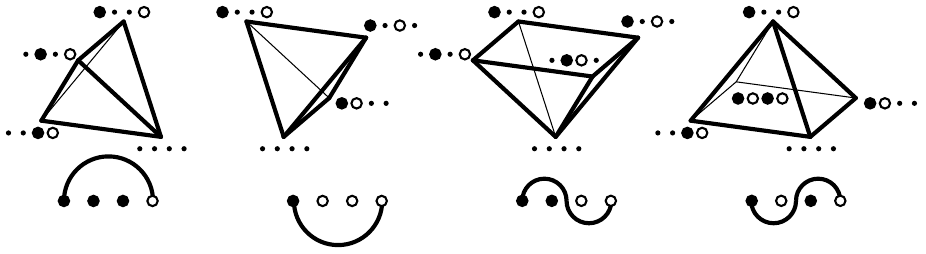}}
	\caption{The shard polytopes of the up arc~$(1, 4, \{2, 3\}, \varnothing)$, the down arc~$(1, 4, \varnothing, \{2, 3\})$, the yin arc~$(1, 4, \{2\}, \{3\})$, and the yang arc~$(1, 4, \{3\}, \{2\})$. The vertices of the shard polytope of an arc~$\alpha \eqdef (a, b, A, B)$ are labeled by the corresponding $\alpha$-alternating matchings, where we use solid dots~$\bullet$ for elements in~$\{a\} \cup A$ and hollow dots~$\circ$ for elements in~$B \cup \{b\}$. The corresponding vertex coordinates are directly read replacing~$\bullet$ by~$1$ and~$\circ$ by~$-1$. For instance, the vertex labeled~${\bullet \cdot \cdot \,\circ}$ has coordinates~$(1,0,0,-1)$. Adapted from~\cite[Fig.~10]{MR4584712}.}
	\label{fig:shardPolytopes}
\end{figure}

A set function $f:2^{[n]}\mapsto \R$ is said to be \defn{submodular} when $f(X \cup Y) + f(X \cap Y) \le f(X) + f(Y)$ for any~$X, Y \subseteq [n]$.
By definition, any quotientope is a deformed permutahedron, and can thus be written as
  \[
  \Bigset{x \in \R^n}{ \sum_{i \in [n]} x_i = f([n]) \text{ and } \sum_{i \in X} x_i \le f(X) \text{ for all } \varnothing \ne X \subseteq [n]},
  \]
where~$f$ is a submodular set function.
We will need the following simple result: Given a collection $\polytope{P}_1, \polytope{P}_2,\ldots ,\polytope{P}_k$ of deformed permutahedra defined by submodular set functions $f_1,f_2,\ldots ,f_k$, their Minkowski sum $\sum_i \polytope{P}_i$ is defined by the submodular function $f \eqdef \sum_i f_i$.


\section{Weak rectangulotopes}
\label{sec:weakRectangulotopes}

This section is devoted to weak rectangulations and the proof of \cref{thm:weakRectangulotope}.


\subsection{The weak poset}
\label{subsec:weakPoset}

Given a rectangulation~$R$, let us consider its source and target trees~$S(R)$ and~$T(R)$, and orient the horizontal edges from left to right, and the vertical edges from bottom to top.
This defines two partial orders on $[n]$.
The \defn{weak poset}~$([n],\prec_w^R)$ of the rectangulation~$R$ is the transitive closure of the union of these two partial orders defined by~$S(R)$ and~$T(R)$.
Equivalently, the weak poset can be defined by considering the adjacency graph of the rectangles in a diagonal representative of the weak equivalence class of $R$, orienting its edges from left to right and from bottom to top, and taking the transitive closure of this directed acyclic graph.
These posets were characterized (and called Baxter posets) by E.~Meehan~\cite{MR4014603}.
The weak poset is a two-dimensional lattice, whose minimum is the root of~$S(R)$ and whose maximum is the root of~$T(R)$.
See \cref{fig:weakPoset}.

\begin{figure}
	\centerline{\includegraphics[width=.4\textwidth]{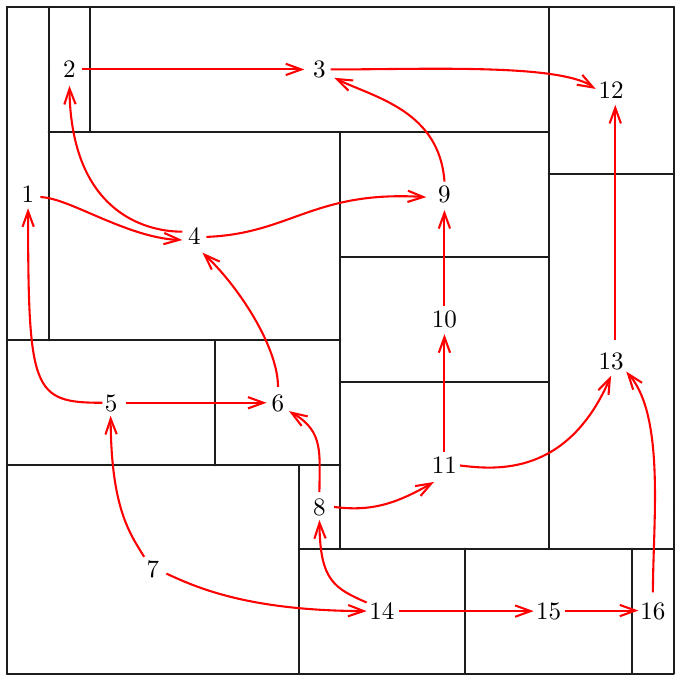} \qquad \includegraphics[width=.4\textwidth]{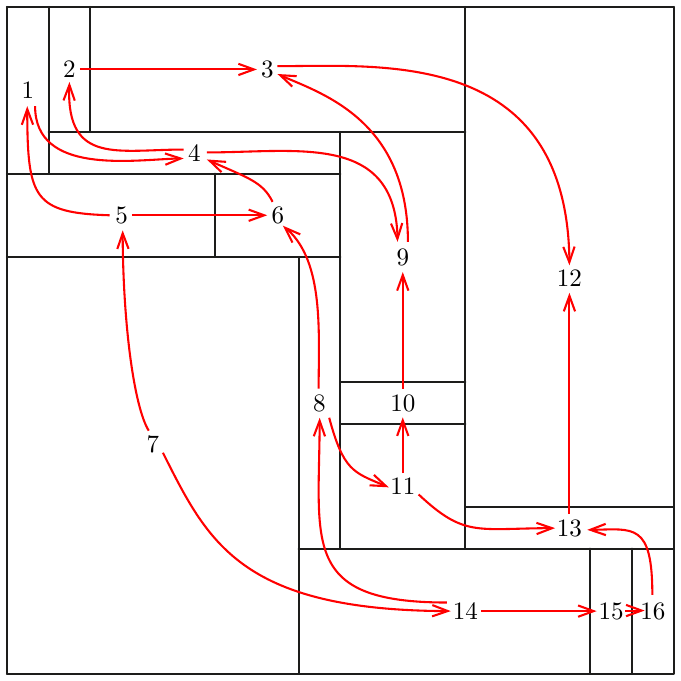}}
	\caption{The weak poset of the rectangulation of \cref{fig:strongRectangulation}, visualized on the rectangulation of \cref{fig:strongRectangulation} (left) and on its diagonal rectangulation of \cref{fig:weakRectangulation} (right). Example from \cite{ACFF24}.}
	\label{fig:weakPoset}
\end{figure}


\pagebreak
\subsection{The weak rectangulation congruence}
\label{subsec:weakRectangulationCongruence}

We consider the set~$\mathcal{L}(\prec_w^R)$ of linear extensions of the weak poset~$\prec_w^R$:
\[
\mathcal{L}(\prec_w^R) \eqdef \bigset{\sigma\in\f{S}_n }{ i\prec_w^R j\implies \sigma^{-1}_i < \sigma^{-1}_j}.
\]

We recall that \defn{Baxter permutations}~\cite{MR0491652,MR0555815} are permutations avoiding the vincular patterns $2\underline{41}3$ and $3\underline{14}2$, where the underlining indicates that the symbols must appear contiguously in the permutation.
Similarly, \defn{twisted Baxter permutations} are permutations avoiding the patterns $2\underline{41}3$ and $3\underline{41}2$, and \defn{co-twisted Baxter permutations} are permutations avoiding the patterns $2\underline{14}3$ and $3\underline{14}2$.
These permutation families are equinumerous and their sizes are the \defn{Baxter numbers}, see \OEIS{A001181}.
They are central in the theory of rectangulations and appear in many other contexts~\cite{MR2763051}.

\begin{theorem}[\cite{MR2871762}]
  \label{thm:weakCong}
  The subsets $\mathcal{L}(\prec_w^R)$ of $\f{S}_n$, for all rectangulations $R$ of size~$n$, 
  are equivalence classes of a congruence relation $\weakeq$ on the weak Bruhat order on $\f{S}_n$.
  Furthermore:
  \begin{itemize}
  \item the congruence classes are one-to-one with weak rectangulations,  
  \item the minimal elements of the congruence classes are the twisted Baxter permutations,
  \item the maximal elements of the congruence classes are the co-twisted Baxter permutations,
  \item the congruence classes each contain a single Baxter permutation. 
  \end{itemize}
\end{theorem}

We call this congruence the \defn{weak rectangulation congruence} (it is sometimes referred to as the \defn{Baxter congruence}).
It is the intersection of the sylvester congruence with the anti-sylvester congruence.
The \defn{weak rectangulotope} $\WRP(n)$ is the quotientope of the weak rectangulation congruence $\weakeq$ on the weak Bruhat order on $\f{S}_n$.

By construction, the quotient fan~$\c{F}_\weakeq$ has a maximal cone~$\polytope{C}_w^R$ for each weak rectangulation~$R$ of size~$n$.
This cone is given by~$\polytope{C}_w^R = \set{\b{x} \in \R^n}{x_i \le x_j \text{ if } i \prec_w^R j}$. 
Note that there is a bijection between the facets of~$\polytope{C}_w^R$, the cover relations of~$\prec_w^R$, and the feasible flips in~$R$.
The definition of the weak poset just translates the fact that the weak rectangulation fan~$\c{F}_\weakeq$ is the common refinement of the sylvester fan and the anti-sylvester fan.
Namely, as~$\polytope{C}_w^R$ is the intersection of the cone of the sylvester fan corresponding to the target tree~$T(R)$ with the cone of the anti-sylvester fan corresponding to the source tree~$S(R)$, the weak poset~$\prec_w^R$ is the transitive closure of the union of the posets associated to~$S(R)$ and~$T(R)$.

Note that weak rectangulotopes were first constructed by S.~Law and N.~Reading in~\cite{MR2871762} as Minkowski sums of two opposite associahedra.
Our vertex and facet descriptions of \cref{thm:weakRectangulotope} can be directly derived from this construction.
Here, we provide a more pedestrian approach based on Minkowski sums of shard polytopes, as it will serve as a template for the more involved proof of \cref{thm:strongRectangulotope} in \cref{sec:strongRectangulotopes}.


\subsection{Up and down arcs}
\label{subsec:upDownArcs}

The arc ideal of the weak rectangulation congruence was already considered in~\cite[Exm.~4.10]{Reading-arcDiagrams}.

\begin{lemma}[{\cite[Exm.~4.10]{Reading-arcDiagrams}}]
The weak rectangulation congruence $\weakeq$ on $\f{S}_n$ is defined by the arc ideal composed of arcs that do not cross the horizontal line.
\end{lemma}
\begin{proof}
  We consider the forbidden vincular patterns $2\underline{41}3$ and $3\underline{41}2$ that define the minimal elements of each congruence class of $\weakeq$.
  We observe that every occurrence of a pattern in a permutation corresponds to an occurrence of an arc that cross the horizontal line at least once in the arc diagram of the permutation.
  The twisted Baxter permutations are therefore the permutations whose arc diagrams are composed only of arcs that do not cross the horizontal line.
\end{proof}

We classified these arcs into two families: the up arcs and the down arcs.
For an interval~$I$ of~$[n]$ of size at least~$2$, the \defn{up arc}~$\upArc{I}$ (resp.~\defn{down arc}~$\downArc{I}$) is the arc that starts at~$\min(I)$ and ends at~$\max(I)$, and passes above (resp.~below) all remaining elements of~$I$, that is,
\begin{align*}
\upArc{I} & \eqdef \big( \min(I), \max(I), I \ssm \{\min(I), \max(I)\}, \varnothing \big) \\
\qquad\text{and}\qquad
\downArc{I} & \eqdef \big( \min(I), \max(I), \varnothing, I \ssm \{\min(I), \max(I)\} \big).
\end{align*}
Note that the basic arc~$(i, i+1, \varnothing, \varnothing)$ is both an up and down arc.
The up arcs correspond to the sylvester congruence, while the down arcs correspond to the anti-sylvester congruence.


\subsection{Shard polytopes of up and down arcs}
\label{subsec:upDownShardPolytopes}

From the results of \cref{subsec:quotientopes}, the quotientope is realized by a Minkowski sum of the shard polytopes of the up arcs and the down arcs.
It immediately follows from the definition that these quotientopes are simplices and anti-simplices.

\begin{lemma}
  \label{lem:lodaysp}
  The up shard polytope $\SP(\upArc{I})$ is the translate of the simplex~$\conv \set{ \b{e}_i }{ i\in I }$ by the vector~$- \b{e}_{\max(I)}$.
\end{lemma}


\begin{lemma}
  \label{lem:antilodaysp}
  The down shard polytope $\SP(\downArc{I})$ is the translate of the negative simplex~$\conv \set{ - \b{e}_i }{ i\in I }$ by the vector~$\b{e}_{\min(I)}$.
\end{lemma}


We refer to \cref{fig:shardPolytopes} for examples of up and down shard polytopes.

\begin{proposition}
  \label{prop:weakMinkowski}
  The weak rectangulotope $\WRP(n)$ is realized by the Minkowski sum of all up and down shard polytopes.
\end{proposition}

Since the combinatorial type of a Minkowski sum is invariant to translation of the summands, we can use the simplices defined in \cref{lem:lodaysp,lem:antilodaysp} as a definition of the up and down shard polytopes.
Also note that since the basic arcs~$(i, i+1, \varnothing, \varnothing)$ are both up and down, each basic shard polytope is summed twice in the proposed realization.
This, again, does not change the combinatorial type of the polytope.

\cref{prop:weakMinkowski} in particular recovers the result of~\cite{MR2871762} that $\WRP(n)$ is the sum of two opposite associahedra of J.-L.~Loday~\cite{MR2871762}, corresponding to the sylvester and anti-sylvester congruences.


\subsection{Submodular functions of weak rectangulotopes}
\label{subsec:submodularWeakRectangulotopes}

The description of the weak rectangulotope~$\WRP(n)$ as a Minkowski sum in \cref{prop:weakMinkowski} allows us to easily compute the corresponding submodular function.
The arc ideal of the weak rectangulation congruence contains exactly all up arcs and all down arcs, hence from the result of~\cite{MR4328906} mentioned in \cref{subsec:quotientFans}, all $2^n-2$ rays of the braid fan are preserved in the fan of the weak rectangulotope.
This implies that every nonempty proper subset will define a facet of the weak rectangulotope, and the description of the weak rectangulotope by its submodular function is actually an irredundant facet description.

\begin{lemma}
  The weak rectangulotope $\WRP(n)$ is realized by the following submodular function:
  \[
  f_{\weakeq}(X) \eqdef \#\bigset{ I\text{ interval of } [n] }{ I \not\subseteq X \text{ and } I \not\subseteq [n]\ssm X }.
  \]
\end{lemma}

\begin{proof}
  The submodular function defining the realization of $\WRP(n)$ is the sum of the submodular functions defining the up and down shard polytope in \cref{lem:lodaysp,lem:antilodaysp}.
  The shard polytope~$\SP(\upArc{I})$ of an up arc~$\upArc{I}$ is given by the submodular function
  \[
  f_{\upArc{I}}(X) \eqdef \llbracket X\cap I\not=\emptyset \rrbracket .
  \]
  (Recall that $\llbracket \varphi\rrbracket$ is equal to 1 if $\varphi$ holds, and~$0$ otherwise.)
  Similarly, the shard polytope~$\SP(\downArc{I})$ of a down arc~$\downArc{I}$ is given by the submodular function
  \[
  f_{\downArc{I}}(X) \eqdef - \llbracket I\subseteq X \rrbracket .
  \]
  Their sum
  \[
  f_{\weakeq}(X) = \sum_{I\text{ interval of }[n]} f_{\upArc{I}}(X) + f_{\downArc{I}}(X)
  \]
  is therefore exactly the number of intervals of $I$ that intersect $X$ but are not contained in $X$, hence the number of intervals contained neither in~$X$ nor in its complement.
\end{proof}

This proves the second part of Theorem~\ref{thm:weakRectangulotope}.

\begin{remark}
Note that the result of~\cite{MR4328906} actually implies that all $2^n-2$ rays of the braid arrangement are preserved in the quotient fan~$\c{F}_\equiv$ if and only if~$\equiv$ refines the weak rectangulation congruence~$\weakeq$.
\end{remark}


\subsection{Loday coordinates of weak rectangulotopes}
\label{subsec:LodayWeakRectangulotopes}

For every weak rectangulation $R$, we wish to give the coordinates of a point $\b{p}(R)\in\R^n$ such that $\WRP(n)$ is the convex hull of the points~$\b{p}(R)$ for all weak rectangulations~$R$ of size~$n$.
Recall that in a Minkowski sum of polytopes, the vertex that is extremal in a generic direction is the sum of the vertices of the summands that are extremal in this direction.
We therefore need to understand which vertices of the translates of the shard polytopes defined in \cref{lem:lodaysp,lem:antilodaysp} are extremal for a direction given by a permutation.

\begin{lemma}
  \label{lem:lodaymax}
  Let $\sigma \in \f{S}_n$ and $I$ be an interval of~$[n]$ of size at least~$2$.
  The unique vertex of the up shard polytope~$\SP(\upArc{I})$ (resp.~down shard polytope~$\SP(\downArc{I})$) that is extreme with respect to the direction $\sigma^{-1}$ is $\b{e}_i$ (resp.~$-\b{e}_i$), where $i$ is the maximum (resp.~the minimum) of~$\sigma^{-1}$ in $I$.
\end{lemma}

\begin{proof}
Recall from \cref{lem:lodaysp,lem:antilodaysp} that the up shard polytope~$\SP(\upArc{I})$ (resp.~down shard polytope~$\SP(\downArc{I})$) is the simplex~$\conv \set{ \b{e}_i }{ i\in I }$ (resp.~the negative simplex~$\conv \set{ - \b{e}_i }{ i\in I }$).
Hence, the optimal vertex of~$\SP(\upArc{I})$ (resp.~$\SP(\downArc{I})$) in direction~$\sigma^{-1}$ is $\b{e}_i$ (resp.~$-\b{e}_i$), where $i$ is the maximum (resp.~the minimum) of~$\sigma^{-1}$ in $I$.
\end{proof}

Note that this extremal vertex should only depend on the weak rectangulation congruence class of $\sigma$.
This, in passing, gives us alternative characterizations of the sylvester congruence~$\frown$, the anti-sylvester congruence~$\smile$, and the weak rectangulation congruence $\weakeq$.
We have~$\sigma \frown \tau$ if and only if $\arg\max_{j\in I} \sigma^{-1}_j=\arg\max_{j\in I} \tau^{-1}_j$ for any interval $I$ of $[n]$.
Similarly, $\sigma \smile \tau$ if and only if $\arg\min_{j\in I} \sigma^{-1}_j=\arg\min_{j\in I} \tau^{-1}_j$ for any interval $I$ of $[n]$.
Finally, $\sigma\weakeq\tau$ if both conditions hold.

The proof of the following lemma actually follows from Loday's realization in~\cref{thm:loday}.
For the sake of completeness, and because we will reuse the same reasoning in \cref{sec:strongRectangulotopes}, we give an elementary proof.

\begin{lemma}
  \label{lem:weakCoord}
  Let $R$ be a rectangulation of size~$n$, with target tree~$T$, and let $\b{p}(R)$ the vertex of~$\WRP(n)$ associated with $R$.
  Given $i\in [n]$, the number of up arcs~$\upArc{I}$ such that $\b{e}_i$ is the extremal vertex of~$\SP(\upArc{I})$ contributing to~$\b{p}(R)$ is
  \[
  \loday{w}_i =  h^{T}_i\cdot v^{T}_i,
  \]
   where $h^T_i$ and $v^T_i$ denote respectively the number of leaves in the horizontal and vertical subtrees of $i$ in the tree $T$.
\end{lemma}

\begin{proof}
  Let $\sigma$ be any permutation in $\mathcal{L}(\prec_w^R)$.
  From \cref{lem:lodaymax}, for an $i\in [n]$, a shard polytope~$\SP(\upArc{I})$ contributes to $\b{e}_i$ if and only if $i$ is the index of the maximum of $\sigma^{-1}$ in $I$.
  We claim that the number of intervals $I$ satisfying this condition is the given product.

  We first show that the number of choices for the left endpoint of $I$ is equal to the number $h^T_i$ of leaves in the horizontal subtree of $i$ in $T$.
  Each such leaf determines a left endpoint for $I$, defined as the node of the tree that follows it in an inorder traversal.
  Any such node $\ell$ either lies in the horizontal subtree of $i$, or is $i$ itself.
  Since the weak poset extends the order induced by $T$, we have that either $\ell=i$ or $\ell\prec_w^R i$, hence $\sigma^{-1}_\ell \leq \sigma^{-1}_i$.
  Therefore, each such choice of left endpoint for $I$ does not contradict the fact that $i$ is the index of the maximum of $\sigma^{-1}$ in $I$.
  Now consider the leaf of $T$ that is immediately left of the leftmost leaf of the horizontal subtree of $i$.
  This leaf is associated with the node $\ell$ that comes next in inorder traversal, which must be an ancestor of $i$ in $T$.
  Therefore, we must have $\ell\succ_w^R i$, and $\sigma^{-1}_\ell > \sigma^{-1}_i$, and we cannot extend $I$ past this index on the left.
  This shows that there are exactly $h^T_i$ choices for the choice of the left endpoint of $I$.
  A symmetric argument shows that there are exactly $v^{T}_i$ choices for the right endpoint of $I$.
  The combined number of choices is thus the product~$h^{T}_i\cdot v^{T}_i$.
\end{proof}

A symmetric analogue of \cref{lem:weakCoord} holds for down shard polytopes, where the product contributes to $-\b{e}_i$ instead of $\b{e}_i$, and we consider the tree $S(R)$ instead of the tree $T(R)$.
This concludes the proof of \cref{thm:weakRectangulotope}.


\section{Strong rectangulotopes}
\label{sec:strongRectangulotopes}

This section is devoted to strong rectangulations and the proof of \cref{thm:strongRectangulotope}.


\subsection{The strong poset}
\label{subsec:strongPoset}

Given a rectangulation $R$, we define the \defn{strong poset} $([n],\prec_s^R)$ of $R$ by considering two relations on $[n]$~\cite{ACFF24}:
\begin{enumerate}
\item The \defn{adjacency poset} $([n],\tri)$ of $R$, defined by considering the adjacency graph of the rectangles of $R$ (identified by their inorder labels in $[n]$), orienting its edges from left to right and from bottom to top, and taking the transitive closure of this directed acyclic graph.
\item The \defn{special relation} $([n],\btri)$ of $R$, in which $i\btri j$ if the two rectangles $r_i$ and $r_j$ with respective inorder labels $i$ and $j$ are such that 
	\begin{enumerate}[(i)]
	\item either the right side of $r_j$ and the left side of $r_i$ lie on the same vertical segment, and the bottom--right corner of $r_j$ is above the top--left corner of $r_i$, 
	\item or the top side of $r_j$ and the bottom side of rectangle $r_i$ lie on the same horizontal segment, and the top--left corner of $r_j$ lies on the right of the bottom--right corner~of~$r_i$.
	\end{enumerate}
\end{enumerate}
The strong poset is the transitive closure of the union of the two relations $\tri$ and $\btri$.
See \cref{fig:strongPoset}.
The strong poset is a two-dimensional lattice, whose minimum is the root of~$S(R)$ and whose maximum is the root of~$T(R)$. The following statement will be used later.

\begin{figure}
	\centerline{\includegraphics[width=.4\textwidth]{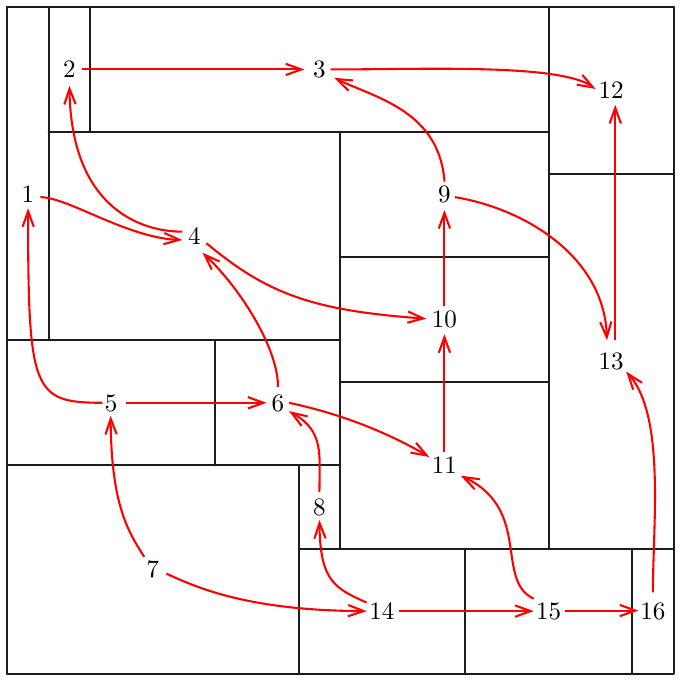} \qquad \includegraphics[width=.4\textwidth]{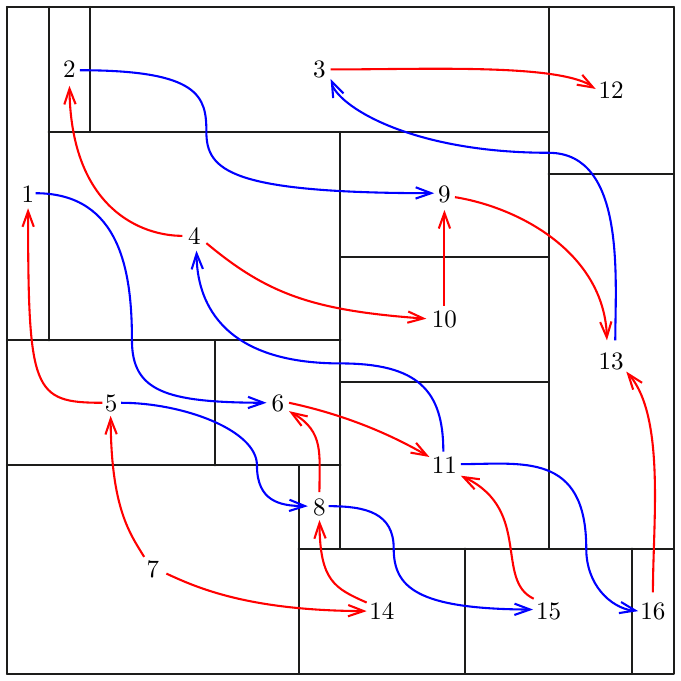}}
	\caption{The adjacency poset (left) and the strong poset (right) of the rectangulation of \cref{fig:strongRectangulation}. Example from \cite{ACFF24}.}
	\label{fig:strongPoset}
\end{figure}

\begin{lemma}
  \label{lem:sextw}
  For any rectangulation $R$, its strong poset $([n], \prec_s^R)$ is an extension of its weak poset~$([n], \prec_w^R)$.
\end{lemma}
\begin{proof}
  Clearly, the strong poset is an extension of the adjacency poset $([n],\tri)$ of $R$, since it is obtained by adding pairs from the special relation~$([n],\btri)$.
  It remains to observe that the adjacency poset of $R$ is an extension of the weak poset of $R$ (compare for instance \cref{fig:weakPoset}\,(left) and \cref{fig:strongPoset}\,(left)).
  For this, we consider a diagonal representative~$D$ of the weak equivalence class of $R$, and recall that the weak poset of $R$ is the adjacency poset of $D$.
  By definition, $R$ can be obtained from $D$ by a sequence of wall slides.
  Each wall slide in this sequence consists of either a horizontal segment being pulled down past another horizontal segment on its right, or a vertical segment being pushed right past another vertical segment below it.
  (These are the flips described in~\cref{fig:wallSlides}, performed in bottom to top order on the figure.) 
  Note that each of these wall slides can only add new pairs in the adjacency poset, hence the adjacency poset of $R$ is always an extension of that of $D$. 
\end{proof}


\subsection{The strong rectangulation congruence}
\label{subsec:strongRectangulationCongruence}

We consider the set~$\mathcal{L}(\prec_s^R)$ of linear extensions of the strong poset~$\prec_s^R$:
\[
\mathcal{L}(\prec_s^R) \eqdef \bigset{\sigma\in\f{S}_n }{ i\prec_s^R j\implies \sigma^{-1}_i < \sigma^{-1}_j}.
\]

N. Reading introduced \defn{2-clumped permutations}~\cite{MR2864445}, defined as the permutations avoiding the vincular patterns $24\underline{51}3$, $42\underline{51}3$, $3\underline{51}24$, and $3\underline{51}42$.
The \defn{co-2-clumped permutations} are defined similarly, as permutations avoiding the mirror images of these four patterns.
These two permutation families are equinumerous, see \OEIS{A342141}.
Reading showed that 2-clumped permutations are in bijection with strong rectangulations.

\begin{theorem}[\cite{MR2864445,ACFF24}]
  \label{thm:strongCong}
  The subsets $\mathcal{L}(\prec_s^R)$ of $\f{S}_n$, for all rectangulations $R$ of size~$n$, are equivalence classes of a congruence relation $\strongeq$ on the weak Bruhat order on $\f{S}_n$.
    Furthermore:
  \begin{itemize}
  \item the congruence classes are one-to-one with strong rectangulations,  
  \item the minimal elements of the congruence classes are the 2-clumped permutations,
  \item the maximal elements of the congruence classes are the co-2-clumped permutations.
  \end{itemize}
\end{theorem}

We call the congruence $\strongeq$ the \defn{strong rectangulation congruence}.
The \defn{strong rectangulotope}~$\SRP(n)$ is the quotientope of the strong rectangulation congruence~${\strongeq}$.

Similarly to the weak case, the quotient fan~$\c{F}_{\strongeq}$ has a maximal cone~$\polytope{C}_s^R$ for each strong rectangulation~$R$ of size~$n$.
This cone is given by~$\polytope{C}_s^R = \set{\b{x} \in \R^n}{x_i \le x_j \text{ if } i \prec_s^R j}$.
Note that there is a bijection between the facets of~$\polytope{C}_s^R$, the cover relations of~$\prec_s^R$, and the feasible flips in~$R$.


\subsection{Yin and yang arcs}
\label{subsec:yinYangArcs}

The arc ideal of the strong rectangulation congruence was already considered in~\cite[Exm.~4.11]{Reading-arcDiagrams}.

\begin{lemma}[{\cite[Exm.~4.11]{Reading-arcDiagrams}}]
  The strong rectangulation congruence $\strongeq$ on $\f{S}_n$ is defined by the arc ideal composed of arcs that cross the horizontal line at most once.
\end{lemma}
\begin{proof}
  We consider the forbidden vincular patterns $24\underline{51}3$, $42\underline{51}3$, $3\underline{51}24$, and $3\underline{51}42$ that define the 2-clumped permutations, the minimal elements of each congruence class of $\strongeq$.
  We observe that every occurrence of a pattern in a permutation corresponds to an occurrence of an arc that crosses the horizontal line at least twice in the arc diagram of the permutation.
  The 2-clumped permutations are therefore the permutations whose arc diagrams are composed only of arcs that cross the horizontal line at most once.
\end{proof}

We classify these arcs into two families, depending on whether the arcs do not cross the horizontal line from below or from above.
For two consecutive intervals~$I$ and~$J$ of~$[n]$, the \defn{yin arc}~$\yinArc{I}{J}$ (resp.~\defn{yang arc}~$\yangArc{I}{J}$) is the arc that starts at~$\min(I)$ and ends at~$\max(J)$, and goes above (resp.~below) the remaining elements of $I$ and below (resp.~above) the remaining elements of $J$, that is,
\begin{align*}
\yinArc{I}{J} & \eqdef \big( \min(I), \max(J), I \ssm \{\min(I)\}, J \ssm \{\max(J)\} \big) \\
\qquad\text{and}\qquad
\yangArc{I}{J} & \eqdef \big( \min(I), \max(J), J \ssm \{\max(J)\}, I \ssm \{\min(I)\} \big).
\end{align*}
Note that all up arcs and down arcs are both yin arcs and yang arcs.

Note that the yin (resp.~yang) arcs alone form an arc ideal, and we call \defn{yin} (resp.~\defn{yang}) \defn{congruence} the corresponding congruence of the weak order.
It can be checked that the equivalence classes of these congruences are counted by \defn{semi-Baxter permutations}, defined as permutations avoiding $2\underline{41}3$, hence only one of the two forbidden Baxter patterns.
These equivalence classes are in turn in bijection with rectangulations that avoid one of the two patterns defining diagonal rectangulations, see~\cite{ACFF24}.
The strong rectangulation congruence is the intersection of the yin and yang congruences.


\subsection{Shard polytopes of yin and yang arcs}
\label{subsec:yinYangShardPolytopes}

From the results of \cref{subsec:quotientopes}, the quotientope is realized by a Minkowski sum of the shard polytopes of the yin arcs and the yang arcs.
These shard polytopes can be described as follows.

\begin{lemma}
  \label{lem:yinsp}
  The yin shard polytope $\SP(\yinArc{I}{J})$ is $\conv \{\b{0}\} \cup \set{\b{e}_i - \b{e}_j}{i \in I, \; j \in J}$.
\end{lemma}

\begin{lemma}
  \label{lem:yangsp}
  The yang shard polytope $\SP(\yangArc{I}{J})$ is the translate of~$\conv \{\b{0}\} \cup \set{\b{e}_j - \b{e}_i}{i \in I,\; j \in J}$ by the vector~$\b{e}_{\min(I)} - \b{e}_{\max(J)}$.
\end{lemma}


We refer to \cref{fig:shardPolytopes} for examples of yin and yang shard polytopes.

\begin{proposition}
  \label{prop:strongMinkowski}
  The strong rectangulotope $\SRP(n)$ is realized by the Minkowski sum of all yin and yang shard polytopes.
\end{proposition}

Since the combinatorial type of a Minkowski sum is invariant to translation of the summands, we can use the polytopes defined in \cref{lem:yinsp,lem:yangsp} as a definition of the yin and yang shard polytopes.
Also note that since the up and down shard polytopes are, up to a translation, both yin and yang, each up and down shard polytope is summed twice (hence all basic shard polytopes are summed four times) in the proposed realization.
This, again, does not change the combinatorial type of the polytope.

\cref{prop:strongMinkowski} gives us an alternative description of $\SRP(n)$ as the sum of two opposite quotientopes corresponding to the yin and yang congruences.


\subsection{Submodular functions of strong rectangulotopes}
\label{subsec:submodularStrongRectangulotopes}

From the results of \cref{subsec:quotientopes}, the submodular function corresponding to our realization of $\SRP(n)$ can be obtained as before by summing the functions defining the yin and yang shard polytopes in \cref{lem:yinsp,lem:yangsp}.
Note again that all rays of the braid fan are preserved in the fan of the strong rectangulotope, hence our submodular description of the strong rectangulotope is an irredundant facet description.

\begin{lemma}
  The strong rectangulotope $\SRP(n)$ is realized by the following submodular function:
  \[
  f_{\strongeq}(X) \eqdef \# \bigset{ I, J }{ I \not\subseteq [n] \ssm X \text{ and } J \not\subseteq X } +
  \#\bigset{ I, J }{ I \not\subseteq X \text{ and } J \not\subseteq [n] \ssm X },
  \]
  where $I, J$ are nonempty consecutive intervals of~$[n]$.
\end{lemma}

\begin{proof}
  The submodular function defining the realization of $\WRP(n)$ is the sum of the submodular functions defining the yin and yang shard polytopes in \cref{lem:yinsp,lem:yangsp}.
  The shard polytope~$\SP(\yinArc{I}{J})$ of a yin arc $\yinArc{I}{J}$ is given by the submodular function
  \[
  f_{\yinArc{I}{J}}(X) \eqdef \llbracket I \not\subseteq [n] \ssm X \text{ and } J \not\subseteq X \rrbracket .
  \]
  (Recall that $\llbracket \varphi\rrbracket$ is equal to 1 if $\varphi$ holds, and~$0$ otherwise.)
  Similarly, the shard polytope~$\SP(\yangArc{I}{J})$ of a yang arc~$\yangArc{I}{J}$ is given by the submodular function
  \[
  f_{\yangArc{I}{J}}(X) \eqdef \llbracket I \not\subseteq X \text{ and } J \not\subseteq [n] \ssm X \rrbracket .
  \]
  Their sum
  \[
  f_{\strongeq}(X) \eqdef \sum_{I,J} f_{\yinArc{I}{J}}(X) + f_{\yangArc{I}{J}}(X)
  \]
  is therefore as claimed.
\end{proof}

This proves the second part of Theorem~\ref{thm:strongRectangulotope}.


\subsection{Loday coordinates of strong rectangulotopes}
\label{subsec:LodayStrongRectangulotopes}

As in \cref{subsec:LodayWeakRectangulotopes}, we need to understand which vertices of the translates of the yin and yang shard polytopes defined in \cref{lem:yinsp,lem:yangsp} are extremal for a direction given by a permutation.

\begin{lemma}
  \label{lem:yinminmax}
  Let $\sigma\in\f{S}_n$, let~$I$ and~$J$ be two consecutive intervals of~$[n]$, and let $i\eqdef \arg\max_{k\in I} \sigma^{-1}_k$ and $j \eqdef \arg\min_{k\in J} \sigma^{-1}_k$ (resp.~$i\eqdef \arg\min_{k\in I} \sigma^{-1}_k$ and $j \eqdef \arg\max_{k\in J} \sigma^{-1}_k$).
  The unique vertex of the yin (resp.~yang) shard polytope~$\SP(\yinArc{I}{J})$ (resp.~$\SP(\yangArc{I}{J})$) that is extreme with respect to the direction $\sigma^{-1}$ is $\b{e}_i-\b{e}_j$ (resp.~$\b{e}_j-\b{e}_i$) if $\sigma^{-1}_i > \sigma^{-1}_j$ (resp.~$\sigma^{-1}_i < \sigma^{-1}_j$), and $\b{0}$ otherwise.
\end{lemma}

Note again that by definition, this extremal vertex should only depend on the strong congruence class of $\sigma$.
This, again, gives us alternative characterizations of the yin congruence $\sim$, the yang congruence $\backsim$, and the strong rectangulation congruence $\strongeq$.
We have $\sigma \sim \tau$ if and only if for any two consecutive intervals $I$ and $J$ of $[n]$, 
$\arg\max_{k\in I} \sigma^{-1}_k = \arg\max_{k\in I} \tau^{-1}_k$ and 
$\arg\min_{k\in J} \sigma^{-1}_k = \arg\min_{k\in J} \tau^{-1}_k$ (that is, if $\sigma \weakeq \tau$),
and, if we let $i$ and $j$ be defined as in Lemma~\ref{lem:yinminmax} as the indices of the maximum in $I$ and the minimum in $J$,
$\sigma^{-1}_i > \sigma^{-1}_j \Leftrightarrow \tau^{-1}_i > \tau^{-1}_j$.
A similar statement holds for the yang congruence $\backsim$, by exchanging $\min$ and $\max$.
Finally, $\sigma\strongeq\tau$ if both sets of conditions hold.

\begin{lemma}
  \label{lem:yincount}
  Let $R$ be a strong rectangulation of size~$n$, let~$S$ and $T$ be its source and target trees, let~$([n],\prec_s^R)$ be its strong poset, and let~$\b{p}(R)$ be the vertex of $\SRP(n)$ corresponding to $R$.
  Given a pair $i,j\in [n]$ with $i<j$, the number of yin arcs $\alpha$ such that $\b{e}_i-\b{e}_j$ is the extremal vertex of~$\SP(\alpha)$ contributing to $\b{p}(R)$ is
  \[
    h^T_i \cdot cv^{T,S}_{i,j}\cdot h^S_j 
  \]
  if $i\succ_s^R j$, and $0$ otherwise, where
  \begin{itemize}
  \item $h^T_i$ denotes the number of leaves in the horizontal subtree of $i$ in~$T$,
  \item $cv^{T,S}_{i,j}$ denotes the number of common leaves of the vertical subtree of $i$ in $T$ and the vertical subtree of $j$ in $S$.
  \end{itemize}  
\end{lemma}
\begin{proof}
  Let $\sigma$ be any permutation in~$\mathcal{L}(\prec_s^R)$.
  From \cref{lem:yinminmax}, for a pair $i<j \in [n]$, there exists a yin arc~$\alpha$ whose shard polytope $\SP(\alpha)$ contributes to $\b{e}_i-\b{e}_j$ if and only if $\sigma^{-1}_i > \sigma^{-1}_j$, hence if $i\succ_s^R j$.
  Provided this holds, the set of yin arcs $\alpha$ such that $\SP(\alpha)$ contributes to $\b{e}_i-\b{e}_j$ are defined by pairs of contiguous intervals $I,J$ of $[n]$ with $i\in I$, $j\in J$, such that $i$ is the index of the maximum of $\sigma^{-1}$ in $I$, and $j$ is the index of the maximum of $\sigma^{-1}$ in $J$.
  We claim that the number of choices of the three endpoints defining such a pair of intervals $I$ and $J$ is the given product.

  We first show that the number of choices for the left endpoint of $I$ is equal to the number $h^T_i$ of leaves in the horizontal subtree of $i$ in $T$.
  The proof of this follows the exact same lines as that of \cref{lem:weakCoord}.
  Each leaf determines a left endpoint for $I$, defined as the next node in an inorder traversal.
  Any such node $\ell$ either lies in the horizontal subtree of $i$, or is $i$ itself.
  Since from \cref{lem:sextw} the strong poset extends the order induced by $T$, we have that either $\ell=i$ or $\ell\prec_s^R i$, hence $\sigma^{-1}_\ell \leq \sigma^{-1}_i$.
  Therefore, each such choice of left endpoint for $I$ does not contradict the fact that $i$ is the index of the maximum of~$\sigma^{-1}$ in $I$.
  Now consider the leaf of $T$ that is immediately left of the leftmost leaf of the horizontal subtree of $i$.
  This leaf is associated with the node $\ell$ that comes next in inorder traversal, which must be an ancestor of $i$ in $T$.
  Therefore, we must have $\ell\succ_s^R i$, and $\sigma^{-1}_\ell > \sigma^{-1}_i$, and we cannot extend $I$ past this index on the left.
  This shows that there are exactly $h^T_i$ choices, and not more, for the choice of the left endpoint of $I$.

  The rest of the formula comes by symmetric arguments.
  First, the right endpoint of $I$ can only be chosen among the leaves of the vertical subtree of $T$ rooted at $i$.
  Then the same reasoning holds for the two endpoints of $J$, replacing $i$ by $j$ and $T$ by $S$.
  Finally, since the right endpoint of $I$ and the left endpoint of $J$ must be consecutive, the selected leaves must correspond to the same pair of successive indices of $[n]$.
  These leaves are exactly the common leaves defined previously.
\end{proof}

\cref{lem:yincount} essentially gives us the coefficient $\yin{w}^R_{i,j}$ in \cref{thm:strongRectangulotope}, except that the condition $\neg \; {}_i \!\!\! \raisebox{.03cm}{\verticalPattern} \!\!\! {}^j$ is replaced by $i\succ_s^R j$.
Recall that ${}_i \!\!\! \raisebox{.03cm}{\verticalPattern} \!\!\! {}^j$ denotes the situation where the rectangles $r_i$ and~$r_j$ of inorder labels $i$ and $j$ share a vertical segment, $r_i$ is on the left of $r_j$, and the bottom edge of $r_i$ is lower than the top edge of $r_j$.
Our aim is to provide a formula for the coordinates of the vertices of the strong rectangulotope that does not involve the computation of the strong poset or one of its linear extensions.
For that purpose, we prove that for every pair $i<j\in [n]$ such that~$i \not\succ_s^R j$, we have either ${}_i \!\!\! \raisebox{.03cm}{\verticalPattern} \!\!\! {}^j$ or $cv^{T,S}_{i,j} = 0$.
Hence, we obtain that the product~$h^T_i \cdot cv^{T,S}_{i,j}\cdot h^S_ j\cdot \llbracket \; \neg \; {}_i \!\!\! \raisebox{.03cm}{\verticalPattern} \!\!\! {}^j \; \rrbracket$ is always~$0$ when~$i \not\succ_s^R j$, so that adding it in \cref{thm:strongRectangulotope} does not harm the result.

\begin{lemma}
  \label{lem:whentosum}
  Let $R$ be a rectangulation of size~$n$, with source and target trees~$S$ and~$T$, and strong order~$\prec_s^R$.
  Given a pair $i,j\in [n]$ such that $i<j$ and $i\not\succ_s^R j$, we~have
  \[
    cv^{T,S}_{i,j} > 0 \quad \Longrightarrow \quad {}_i \!\!\! \raisebox{.03cm}{\verticalPattern} \!\!\! {}^j.
  \]
\end{lemma}
\begin{proof}
  Consider the two rectangles $r_i$ and $r_j$ with inorder labels $i$ and $j$ in $R$.
  We first observe that the source vertex $s(r_j)$ is not contained in the interior of the lower left quadrant of origin~$t(r_i)$.
  Indeed, if we assume otherwise, then the line segment from $s(r_j)$ to $t(r_i)$ has positive slope, which implies that $j\tri i$, which in turn implies $j\prec_s^R i$, contradicting our hypothesis.
  We also observe that since $i < j$, by definition of the inorder on~$S$ and~$T$, the rectangle $r_j$ must be below or on the right of~$r_i$.
  Combining those two observations, we conclude that the right edge of $r_i$ cannot be on the right of the left edge of~$r_j$.

  Therefore, there can only be one segment which may contain a leaf to the vertical subtree of $i$ in $T$ and a leaf of the vertical subtree of $j$ in $S$, which is the vertical segment supporting the right edge of $r_i$ and the left edge of $r_j$.
  Let us assume that this segment is shared by those two edges.
  If the bottom edge of $r_i$ is higher than the top edge of $r_j$, we must have $j\btri i$, which implies $j\prec_s^R i$, contradicting our hypothesis again.
  Thus the bottom edge of $r_i$ is lower than the top edge of $r_j$, and ${}_i \!\!\! \raisebox{.03cm}{\verticalPattern} \!\!\! {}^j$, as claimed.
\end{proof}

\cref{lem:yincount} and \cref{lem:whentosum} together already provide a realization of the quotientopes of the yin congruence, the vertices of which are in bijection with rectangulations that avoid one of the two patterns defining diagonal rectangulations~\cite{ACFF24}, and with $2\underline{41}3$-avoiding permutations, as mentioned in \cref{subsec:yinYangArcs}.

We omit the proofs of the analogues of \cref{lem:yincount} and \cref{lem:whentosum} for the case of the yang arcs, which use completely symmetric arguments.
This concludes the proof of \cref{thm:strongRectangulotope}.


\addtocontents{toc}{ \vspace{.1cm} }
\bibliographystyle{alpha}
\bibliography{rectangulotopes}
\label{sec:biblio}

\end{document}